\documentclass[reqno]{amsart}
\usepackage{amsfonts,amsmath,amssymb,amsthm,bm,upgreek}
\usepackage{mathrsfs,mathtools}
\usepackage{color}
\usepackage{enumitem}
\usepackage{derivative,physics,longtable}
\usepackage[english]{babel}
\usepackage[colorlinks = true,
linkcolor = green,
urlcolor  = blue,
citecolor = red,
anchorcolor = red]{hyperref}
\usepackage{cleveref}
\usepackage{dutchcal}
\usepackage{accents}
\def\A{\mathcal{A}}
\def\T{\mathcal{T}}
\def\D{\mathcal{D}}

\def\K{\mathcal{K}}
\def\N{\mathcal{N}}
\def\R{\mathbb{R}}

\def\j{\mathbcal{j}}
\def\Lscr{\mathscr{L}}
\def\O{\mathcal{O}}
\def\o{\mathcal{o}}
\def\eps{\varepsilon}
\def\geq{\geqslant}
\def\leq{\leqslant}
\def\grad#1{\nabla #1 }
\def\jp#1{\left<#1\right>  }
\def\odif#1{\mathrm{d}#1\,}
\def\fct#1#2{ #1\left(#2\right) }

\def\Riesz#1{\mathcal{I}_{\lambda}\left[#1\right]}
\def\Laplace{\mathop{}\!\mathbin\bigtriangleup}
\DeclarePairedDelimiterX\setc[2]{\{}{\}}{\,#1 \;\delimsize\vert\; #2\,}
\DeclareMathOperator{\diag}{diag}

\newcommand\munderbar[1]{%
	\underaccent{\bar}{#1}}

\theoremstyle{plain}
\newtheorem{theo}{Theorem}[section]
\newtheorem{lemm}[theo]{Lemma}
\newtheorem{prop}[theo]{Proposition}
\newtheorem{coro}[theo]{Corollary}

\theoremstyle{definition}

\theoremstyle{remark}

\newtheorem{claim}{Claim}

\DeclareUnicodeCharacter{394}{$\Delta$}
\DeclareUnicodeCharacter{2212}{\textendash}
\numberwithin{equation}{section}
\everymath{\displaystyle}
\begin{document}
\title[]{A perturbation result for the energy critical Choquard equation in $\mathbb{R}^N$}
\author[]{Xinyu Bo}
\address{\hskip-1.15em Xinyu Bo
	\hfill\newline School of Mathematics and Statistics, \hfill\newline Nanjing Univeristy of Information Science and Technology,
	\hfill\newline Nanjing, 210044,  People's Republic of China.}
\email{bxinyhuan@163.com}
\author[]{Guangying Lv}
\address{\hskip-1.15em Guangying Lv
	\hfill\newline School of Mathematics and Statistics, \hfill\newline Nanjing Univeristy of Information Science and Technology,
	\hfill\newline Nanjing, 210044,  People's Republic of China.}
\email{gylvmaths@nuist.edu.cn}
\author[]{Xingdong Tang}
\address{\hskip-1.15em Xingdong Tang
	\hfill\newline School of Mathematics and Statistics, \hfill\newline Nanjing Univeristy of Information Science and Technology,
	\hfill\newline Nanjing, 210044,  People's Republic of China.}
\email{txd@nuist.edu.cn}

\author[]{Guixiang Xu}
\address{\hskip-1.15em Guixiang Xu
	\hfill\newline Laboratory of Mathematics and Complex Systems,
	\hfill\newline Ministry of Education,
	\hfill\newline School of Mathematical Sciences,
	\hfill\newline Beijing Normal University,
	\hfill\newline Beijing, 100875, People's Republic of China.}
\email{guixiang@bnu.edu.cn}
\subjclass[2010]{35B25,35B09}
\begin{abstract}
We study the singularly perturbed nonlinear energy critical Choquard equation
\begin{equation*}
-{\Laplace u}\qty({x})
-{\alpha}
\int_{\R^N}\frac{u^p\qty(y)}{\abs{x-y}^{\lambda}}\odif{y}
u^{p-1}\qty({x})
-\eps k\qty(x)u^{\frac{N+2}{N-2}}\qty(x)=0, \qquad x\in\R^N,	
\end{equation*}
where $N\geq 3$, $0<\lambda<N$, $\lambda\leq 4$, $p=\frac{2N-\lambda}{N-2}$, 
$
\alpha =
\frac{
	N\qty({N-2})\fct{\Gamma}{N-\frac{\lambda}{2}}
}{
	\pi^{\frac{N}{2}}\fct{\Gamma}{\frac{N-\lambda}{2}}
}$,~
and
$k$ is a positive function. By making use of a
Lyapunov-Schmidt reduction argument,
for sufficiently small $\eps>0$, we construct solutions of the form  
\begin{align*}
u_{\eps}\qty(x)=U_{\mu_{\eps},\xi_{\eps}}\qty(x)\qty(1+\O\qty(\eps)),	
\end{align*}
where $U_{\mu_{\eps},\xi_{\eps}}$ is a positive solution of the unperturbed equation  
\begin{equation*}
	-{\Laplace u}\qty({x})
	-{\alpha}
	\int_{\R^N}\frac{u^p\qty(y)}{\abs{x-y}^{\lambda}}\odif{y}=0,\qquad x\in\R^N.	
\end{equation*}
\end{abstract}
\keywords{Choquard equation; Hartree equation; Lyapunov-Schmidt reduction; Nonlocal problem; Perturbation} 
\maketitle
\section{Introduction}
In this paper, we consider the singularly perturbed semilinear elliptic equation with critical nonlinear terms
\begin{equation}
\label{NLH-eps}
-{\Laplace u}\qty({x})
-{\alpha}
\int_{\R^N}\frac{u^p\qty(y)}{\abs{x-y}^{\lambda}}\odif{y}
u^{p-1}\qty({x})
-\eps k\qty(x)u^{\frac{N+2}{N-2}}\qty(x)=0, \qquad x\in\R^N,
\end{equation}
where $N\geq 3$, $N\geq 3$, $0<\lambda<N$, $\lambda\leq 4$, 
$
\alpha =
\frac{
	N\qty({N-2})\fct{\Gamma}{N-\frac{\lambda}{2}}
}{
	\pi^{\frac{N}{2}}\fct{\Gamma}{\frac{N-\lambda}{2}}
}$, and
$k$ is a positive function. Throughout this paper, we always take
$p=\frac{2N-\lambda}{N-2}$.

The problem \eqref{NLH-eps} can be seen as a
perturbation of the following standard critical Choquard equation
\begin{align}
\label{NLH0}
-{\Laplace u}\qty({x})
-{\alpha}
\int_{\R^N}\frac{u^p\qty(y)}{\abs{x-y}^{\lambda}}\odif{y}
u^{p-1}\qty({x})
=0,	\qquad x\in\R^N,
\end{align}
which arises in the study of the dynamics of  electrons in the dielectric polarisable continuum, see, for example,
\cite{Frohlich1954,Pekar2022}. It is worthwhile to notice that, 
under suitable regular conditions, \eqref{NLH0} is equivalent to the following integral system
\begin{align}
\label{NLHsys}
\begin{cases}
	u\qty(x)=c_1\int_{\R^N} \frac{v\qty(y)u^{p-1}\qty(y)\odif{y}}{\abs{x-y}^{N-2}},& x\in\R^N,\\
	v\qty(x) = c_2\int_{\R^N}\frac{u^{p}\qty(y)}{\abs{x-y}^{\lambda}}\odif{y}& x\in\R^N,
\end{cases}		
\end{align}
where $c_1$, $c_2$ are positive constants, see \cite{CLO2006,Lei2013,DY2019}.
When $\lambda=N-2$, by making use of the moving plane method and the Kelvin transformation, Lei\cite{Lei2018} showed that any positive solution of \eqref{NLH0} must have the following form
\begin{align}
	\label{sltn:form}
U_{\mu,\xi}\qty(x)=\frac{1}{\qty( \mu+\frac{\abs{x-\xi}^2}{\mu} )^{\frac{N-2}{2}}}, \qq{where} \qty(\mu,\xi)\in\qty(0,+\infty)\times\R^N.	
\end{align}
Later, Du and Yang \cite{DY2019} 
extended the above results to the case $0<\lambda<N$, if~$N=3,4$, while $0<\lambda\leq 4$, if $N\geq 5$. Meanwhile, by combining the moving plane method with the Kelvin transformation
and the argument of Li and Zhu\cite{LiZhu1995}, 
Guo, Xu, Peng and Shuai\cite{GHPS2019} also proved that, for all $\lambda$ satisfying $0<\lambda<N$, if~$N=3,4$, while $N-4\leq \lambda< N$, if $N\geq 5$, any positive solution of \eqref{NLH0} must have the form \eqref{sltn:form}.

Before turning to the problem \eqref{NLH-eps}, let us mention the  singularly perturbed elliptic problems with critical local nonlinear terms
\begin{align}
\label{NLSeps}
-\Laplace u\qty(x)=N\qty(N-2) u^{\frac{N+2}{N-2}}\qty(x)+\eps k\qty(x)u^{\frac{N+2}{N-2}}\qty(x),\qq{}x\in\R^N,
\end{align} 
which appears in the prescribing scalar curvature problem, see e.g \cite{Nirenberg1975,AGP1999}. 
Let $g_0=\sum_{i=1}^N\odif{x^i}\otimes\odif{x^i}$ be the standard metric on $\R^N$, if $u$ is a positive solution to \eqref{NLSeps}, then $g=u^{\frac{4}{N-2}}g_0$ is a metric which is conformal to $g_0$ pointwisely, and $\eps k$ is the scalar curvature corresponding to the metric $g$. Notice that, as $\eps\to 0$, the equation \eqref{NLSeps} recovers to the following unperturbed elliptic problem
\begin{align}
\label{NLS}
-\Laplace u\qty(x)=N\qty(N-2)u^{\frac{N+2}{N-2}}\qty(x),\qq{}x\in\R^N,
\end{align}
which possesses a family of positive solutions 
\begin{align*}
U_{\mu,\xi}\qty(x)=\frac{1}{\qty( \mu+\frac{\abs{x-\xi}^2}{\mu} )^{\frac{N-2}{2}}}, \qq{where} \qty(\mu,\xi)\in\qty(0,+\infty)\times\R^N.	
\end{align*}
Therefore, one may expect that, for $\eps>0$ sufficiently small, \eqref{NLSeps} has at least one positive solution with the form
\begin{align*}
u\qty(x)=U_{\mu_{\eps},\xi_{\eps}}\qty(x)\qty(1+\O\qty(\eps)),
\end{align*}  
for some $\qty(\mu_{\eps},\xi_{\eps})\in\qty(0,+\infty)\times\R^N$.
In \cite{Li1995}, under some suitable algebraic conditions on the first order term of the Taylor expansion to the function $S=1+\eps k$ near the critical points of $S$, Li obtained a positive solution of \eqref{NLSeps}. A few years later, in \cite{AGP1999}, by making use of an abstract finite dimensional reduction argument,  Ambrosetti, Garcia Azorero and Peral showed that \eqref{NLSeps} admits a positive solution under some non degenerate and decay conditions on the function $k$. After then, there has been extensive literature on the topic of existence of positive solutions to \eqref{NLSeps}, see \cite{AGP1999,AM2006,CNY2002,WeiYan2010,LWX2018,Yan2000,GMPY2020,AM1999,LPZ2023,DMW2022,Prashanth2007,PW2016} and references therein.

\subsection{Statement of the main result}
Now, let us focus on the problem \eqref{NLH-eps}, which can be seen as a nonlocal analogue of \eqref{NLSeps}. 
Notice that, the family of positive functions  \eqref{sltn:form} solves \eqref{NLH0}, one may also expect that
\eqref{NLH-eps} has at least one positive solution with the form
\begin{align*}
	u\qty(x)=U_{\mu_{\eps},\xi_{\eps}}\qty(x)\qty(1+\O\qty(\eps)),
\end{align*}  
for some $\qty(\mu_{\eps},\xi_{\eps})\in\qty(0,+\infty)\times\R^N$.
In the present paper, we are mainly interested in finding positive solutions of \eqref{NLH-eps} with the above form. In order to do this, let us make the following assumptions on the function $k$:
\crefname{enumi}{point}{} 
\crefrangelabelformat{enumi}{#3#1#4~~~~#5#2#6} 
\begin{enumerate}[
label={(k.\arabic*)}, 
ref=(k.\arabic*),
leftmargin=0.9cm, 
start=0]
\item\label{k:pstv} $ 0<\inf_{x\in\R^N}k\qty(x)\leq \sup_{x\in\R^N}k\qty(x)<+\infty $. 
\item\label{k:reg} $k\in C^2\qty(\R^N)$.
\item\label{k:critical} The set 
$\mathrm{Cr}\qty[k]= \setc{x\in\R^N}{\grad{k}\qty(x)=0}$ is finite, and 
$\Laplace k\qty(x)\neq 0$, for any $x\in \mathrm{Cr}\qty[k].$
Moreover,
\begin{align}
	\sum_{x\in \mathrm{Cr}\qty[k],\Laplace k\qty(x)<0}i\qty( \grad{k}, x )\neq (-1)^N,
\end{align}
where $i$ denote the fixed-point index.
\item\label{k:ngtv} There exists $\rho>0$ such that $x\cdot\grad{k}\qty(x)<0$, for any $\abs{x}\geq \rho$.
\item\label{k:L1} $x\cdot\grad{k}\qty(x)\in L^1\qty(\R^N)$, and $\int_{\R^N}x\cdot\grad{k}\qty(x)\odif{x}<0.$			
\end{enumerate}
Our main result can be stated as follows.
\begin{theo}
\label{thm:main}
Let $k$ satisfies {\ref{k:pstv}-\ref{k:L1}}. 
Then for any $\eps>0$ sufficiently small, there exists a positive solution to \eqref{NLH-eps} with the form
\begin{align}
	\label{NLH:eps:sltn}
	u_{\eps}\qty(x)
	=U_{\mu_{\eps},\xi_{\eps}}\qty(x)
	+\phi_{\eps}\qty(x)
\end{align} 
for some $\mu_{\eps}\in \qty(0,+\infty)$ and $\xi_{\eps}\in\R^N$.
Moreover, we have
\begin{align*}
\lim_{\eps\to 0}\norm{\frac{\phi_{\eps}}{U_{\mu_{\eps},\xi_{\eps}}}}_{L^\infty}=0.
\end{align*}	
\end{theo}
Let us make several comments on the above result.

{\itshape 1. On the assumptions of $k$.} The assumption \ref{k:pstv} ensures that we can study the positive solutions of \eqref{NLH-eps} by making use of the maximum principle, see \Cref{sec:ansatz} for more details. To our best knowledge, 
the kind of assumptions \ref{k:reg}-\ref{k:L1} was firstly appeared in \cite{AGP1999}, see also \cite{AM1999,Prashanth2007} for example. The assumptions \ref{k:reg} and\ref{k:critical} yield that
every critical point of $k$ is non-degenerate, which enables us to define the fixed-point index of $k$. The technical assumptions \ref{k:ngtv} and\ref{k:L1} ensures that the reduced finite dimensional problem is solvable, and therefore the problem \eqref{NLH-eps} admits at least one positive solution, see \Cref{sec:pf} for more details. 

{\itshape 2. Comparison with critical Choquard equations with an external potential.} This theorem as stated above can be compared with that of the following critical Choquard equations with potential terms
\begin{align}
\label{NLH:potential}
	-{\Laplace u}\qty({x})+V\qty(x)u\qty(x)
	-{\alpha}
	\int_{\R^N}\frac{u^p\qty(y)}{\abs{x-y}^{\lambda}}\odif{y}
	u^{p-1}\qty({x})
	=0,	\qquad x\in\R^N,
\end{align}
which has been extensively studied, see e.g. \cite{GLMY2021,DGY2020,GY2018,GMYZ2022,GDYZ2020,GLMY2021,GHPS2019,MV2017,MV2015,SGY2018} and references therein. 
Particularly, in \cite{MV2015}, by combining a concentration compactness argument with a nonlocal version of Brezis-Lieb lemma, Moroz and Van Schaftingen obtained the existence of a non trivial solution to \eqref{NLH:potential} under the assumption that
\begin{align*}
	\liminf_{\abs{x}\to\infty}\qty( 1-V\qty(x) )\abs{x}^2>\frac{N^2\qty(N-2)}{4\qty(N+1)}.
\end{align*}
A few years later, in \cite{GHPS2019}, by making use of a variational approach developed in \cite{BC1990},
Guo, Hu, Peng and Shuai showed that \eqref{NLH:potential} admits at least one positive solution under the assumption that
\begin{align*}
	0<\qty( \int_{\R^N}\abs{V\qty(x)}^{\frac{N}{2}}\odif{x} )^{\frac{2}{N}}<\qty( 2^{\frac{N+2-\lambda}{2N-\lambda}}-1 )S,
\end{align*}
where $S$ is the best Sobolev constant for the embedding $\dot{H}^1\qty(\R^N)\hookrightarrow L^{\frac{2N}{N-2}}\qty(\R^N)$.
Recently, in \cite{GMYZ2022}, by using a finite dimensional reduction and developing local Poho\v{z}aev identities, Gao, Moroz, Yang and Zhao constructed infinitely many solutions of \eqref{NLH:potential} in $\R^6$ under the assumptions that
\begin{align*}
	i\qty( \grad{\qty(r^2 V\qty(r,x''))}, \qty(r_0,x''_0) )\neq 0,
\end{align*}
where $x=\qty(x',x'')\in\R^2\times\R^4$, $V\qty(x)=V\qty(\abs{x'},x'')$, $\qty(r_0,x''_0)$ is the critical point of the function $r^2 V\qty(r,x'')$ with $r_0>0$ and 
$V\qty(r_0,x_0'')>0$. 
However, as far as we know, there are few papers concerned with the existence of solutions to the problem \eqref{NLH-eps}, which is the main purpose in this manuscript. 

{\itshape 3. Comparison with critical Choquard equations with a nonlocal singularly perturbed terms.} It would be more interesting to study the following critical Choquard equations
\begin{align*}
-{\Laplace u}\qty({x})
-{\alpha}
\int_{\R^N}\frac{u^p\qty(y)}{\abs{x-y}^{\lambda}}\odif{y}
u^{p-1}\qty({x})
-\eps\int_{\R^N}\frac{k\qty(y) u^p\qty(y)}{\abs{x-y}^{\lambda}}\odif{y}
k\qty(x)u^{p-1}\qty({x})
=0,	~~ x\in\R^N.
\end{align*}
In order to study the existence of solutions to the above equation,
instead of \ref{k:pstv}-\ref{k:L1}, new assumptions on $k$ must be proposed. This will be the subject of our forthcoming works.

\subsection{Notation and function spaces}
With the usual abuse of notation, throughout this paper, 
we use the same notation $0$ to stand for the origin of $\R$ and $\R^N$ respectively. 
We use $C$ to denote various finite constants (which may depend on $N$ and $\lambda$). 
If $A$ and $B$ are two quantities with $B$ non-negative, 
We use 
$
A=\O\qty(B)
$  
to mean that $-CB \leq A\leq C B$ for some constant $C$. 

For any $\qty(\mu,\xi)\in\qty(0,+\infty)\times\R^N$, and any function $\varphi:\R^N\to\R$,  we denote
\begin{gather*}
\varphi_{;\mu,\xi}\qty(x)=
\frac{1}{\mu^{\frac{N-2}{2}}}\varphi\qty(\frac{x-\xi}{\mu}),
\qq{and}
\varphi^{\mu}_{;\xi}\qty(x)=
\frac{1}{\mu^{\frac{N+2}{2}}}\varphi\qty(\frac{x-\xi}{\mu}).		
\end{gather*}
For any $\lambda\in\qty(0,N)$ and any function $f:\R^N\mapsto\R$ locally integrable, we denote the $\lambda$-order Riesz potential of the function $f$ by
\begin{equation*}
	\Riesz{f}\qty(x) = 
	\int_{\R^N}
	{ \frac{ f\qty({y})}{\abs{\,x-y\,}^{\lambda}} }\odif{y}.	
\end{equation*}

Now, let us define
\begin{align*}
	\jp{x}=\sqrt{1+\abs{x}^2}.
\end{align*}
Recall that, by  \eqref{sltn:form}, we have
\begin{align}
	U_{\mu,\xi}\qty(x)=
\frac{1}{\mu^{\frac{N-2}{2}}}	
	\frac{1}{\jp{\frac{x-\xi}{\mu}}^{N-2}}.	
\end{align}
For the case $\qty(\mu,\xi)=\qty(1,0)$ in the above expression,
we use a shorthand notation $U=U_{\mu,\xi}$, i.e.
\begin{align}
\label{U}
	U\qty(x)=\frac{1}{\jp{x}^{N-2}}.
\end{align}

The following notation will also be of great use in the sequel.
\begin{gather}
\label{Zj:def}
Z_{0}\qty(x)=
\left.\pdv{U_{\mu,\xi}}{\mu}
\qty(x)
\right|_{\qty(\mu,\xi)=\qty(1,0)},
Z_{j}\qty(x)=
\left.\pdv{U_{\mu,\xi}}{\xi_j}
\qty(x)
\right|_{\qty(\mu,\xi)=\qty(1,0)},	
\\
\label{Zbar}
\bar{Z}_{m,{0}}\qty(x)
=
\left.\pdv{Z_{m;\mu,\xi}}{\mu}
\qty(x)
\right|_{\qty(\mu,\xi)=\qty(1,0)},~~~~
\bar{Z}_{m,j}\qty(x)
=
\left.\pdv{Z_{m;\mu,\xi}}{\xi_j}
\qty(x)
\right|_{\qty(\mu,\xi)=\qty(1,0)},
\\
\label{Hj}
\fct{H_m}{x}=\fct{U^{\frac{4}{N-2}}}{x}\fct{Z_{m}}{x},
\\
\label{Htilde}
\widetilde{H}_{m,0}\qty(x)
=
\left.\pdv{H_{m;\xi}^{\mu}}{\mu}
\qty(x)
\right|_{\qty(\mu,\xi)=\qty(1,0)}
,~~~~
\widetilde{H}_{m,j}\qty(x)
=
\left.\pdv{H_{m;\xi}^{\mu}}{\xi_j}
\qty(x)
\right|_{\qty(\mu,\xi)=\qty(1,0)},		
\end{gather}
where $j=1,2,\ldots,N$ and $m=0,1,\cdots,N$.
By direct computations, we easily see that
\begin{align}
\label{Zj}
\fct{Z_{0}}{x}
=\frac{N-2}{2}\frac{\abs{x}^{2}-1}{\jp{x}^N}
,\quad
\fct{Z_{j}}{x}
=
\qty({N-2})\frac{x_j}{\jp{x}^N},
~
j=1,2,\ldots,N,	
\end{align}

Next, we introduce several Banach spaces. We use $\qty(\dot{H}^1,~\norm{\cdot})$ denote the Sobolev space
\begin{align*}
\dot{H}^1=
\setc*{\phi\in L^{\frac{2N}{N-2}}\qty(\R^N)}{ \nabla \phi\in L^2\qty(\R^N)}	
\end{align*}
with the norm
\begin{align*}
\norm{{\phi}} = \qty(\int_{\R^N}\abs{\nabla\phi\qty(x)}^2\odif{x})^{\frac{1}{2}}.	
\end{align*}

We also use $\qty(X,~\norm{\cdot}_X)$ and $\qty(Y,~\norm{\cdot}_Y)$ 
stand for the Banach spaces
\begin{gather*}
X=\setc*{\phi\in C\qty(\R^N)}{ \sup_{x\in\R^N}\abs{\jp{x}^{N-2}\phi\qty(x)}<+\infty },	
\end{gather*}
and
\begin{gather*}
Y=\setc*{g\in C\qty(\R^N)}{ \sup_{x\in\R^N}\abs{\jp{x}^{N+2}g\qty(x)}<+\infty<+\infty },	
\end{gather*}
with the norms
\begin{gather*}
\norm{\phi}_X = \sup_{x\in\R^N}\abs{\jp{x}^{N-2}\phi\qty(x)},
\qq{and}
\norm{g}_Y = \sup_{x\in\R^N}\abs{\jp{x}^{N+2}g\qty(x)},	
\end{gather*}
respectively.
Moreover, for any $\qty(\mu,\xi)\in\qty(0,+\infty)\times\R^N$, we denote by $\qty(X_{\mu,\xi},~\norm{\cdot}_{X_{\mu,\xi}})$ and $\qty(Y_{\mu,\xi},~\norm{\cdot}_{Y_{\mu,\xi}})$ the Banach spaces
\begin{equation*}
X_{\mu,\xi}=\setc*{\phi\in C\qty(\R^N)}{ \norm{\phi}_{X_{\mu,\xi}}<+\infty },			
\end{equation*}
and
\begin{equation*}
Y_{\mu,\xi}=\setc*{g\in C\qty(\R^N)}{ \norm{g}_{Y_{\mu,\xi}}<+\infty },				
\end{equation*}
with the norms
\begin{align*}
\norm{\phi}_{X_{\mu,\xi}}=\sup_{x\in\R^N}
\mu^{\frac{N-2}{2}}\jp{{\frac{x-\xi}{\mu}}}^{N+2}\abs{\phi\qty(x)},	
\end{align*}
and
\begin{align*}
	\norm{g}_{Y_{\mu,\xi}}=\sup_{x\in\R^N}
	\mu^{\frac{N+2}{2}}\jp{{\frac{x-\xi}{\mu}}}^{N+2}\abs{g\qty(x)},	
\end{align*}
respectively.

\subsection{The ansatz and strategy of the proof}
\label{sec:ansatz}
Let us briefly sketch the main ingredients of the proof of \Cref{thm:main}.

First of all, we claim that, for any $k$ satisfies \ref{k:pstv} and \ref{k:reg}, in order to find non-trivial solutions of \eqref{NLH-eps}, we only need to solve 
\begin{equation}
	\label{NLH-eps+}
	-{\Laplace u}
	-{\alpha}
	\Riesz{\qty(u_{+})^p}
	\qty(u_{+})^{p-1}
	-\eps k\qty(u_{+})^{\frac{N+2}{N-2}}
	=0,\qq{in}\R^N,\quad u\in\dot{H}^1\setminus\{0\},
\end{equation}
where $u_{+}=\max\left\{u,0\right\}$.
Indeed, if $u$ satisfies \eqref{NLH-eps+}, then by multiplying 
$u_{-}=-\min\left\{u,0\right\}$ and integrating over $\R^N$, we get
\begin{align*}
	\int_{\R^N}\abs{\grad{u_{-}}}^2=0,
\end{align*} 
which implies that $u_{-}=0$, and therefore $u\geq 0$. Then by the maximum principle, we immediately get $u>0$. Moreover, by a standard elliptic regularity argument, for any $u$ solves \eqref{NLH-eps+}, we have $u\in C^2\qty(\R^N)$.

Now, by noticing that, for any $\qty(\mu,\xi)\in\qty(0,+\infty)\times\R^N$, 
\begin{equation*}
	-{\Laplace U_{\mu,\xi}}
	-{\alpha}
	\Riesz{U_{\mu,\xi}^p}
	U_{\mu,\xi}^{p-1}
	=0,
\end{equation*}
one may expect that, for $\eps>0$ sufficiently small,  \eqref{NLH-eps+} has at least one solution which
has the form $u=U_{\mu,\xi}+\phi$ with $\phi$ sufficiently small compared with $U_{\mu,\xi}$. 
By direct computations, we get that $u=U_{\mu,\xi}+\phi$ solves \eqref{NLH-eps+} if and only if
$\phi$ satisfies
\begin{align}
	\label{phi:equation}
	{\mathscr{L}_{\mu,{\xi}}\phi}=N\qty(\phi;\mu,\xi)+\eps k\qty(U_{\mu,\xi}+\phi)_{+}^{\frac{N+2}{N-2}},
\end{align}
where
\begin{equation*}
	{\mathscr{L}_{\mu,\xi}\phi}
	=
	-{\Laplace\phi}
	-\alpha p\Riesz{U_{\mu,\xi}^{p-1}\phi}{U_{\mu,\xi}^{p-1}}
	-\alpha\qty(p-1)\Riesz{U_{\mu,\xi}^{p}}{U_{\mu,\xi}^{p-2}}{\phi},	
\end{equation*}
and
\begin{align*}
	N\qty(\phi;\mu,\xi)=
	&{\alpha}
	\Riesz{\qty(U_{\mu,\xi}+\phi)_{+}^p}
	\qty(U_{\mu,\xi}+\phi)_{+}^{p-1}-
	\alpha p\Riesz{U_{\mu,\xi}^{p-1}\phi}{U_{\mu,\xi}^{p-1}}
	\\
	&-\alpha\qty(p-1)\Riesz{U_{\mu,\xi}^{p}}{U_{\mu,\xi}^{p-2}}{\phi}.	
\end{align*}

We solve \eqref{phi:equation} in two steps.
\subsubsection*{Step 1}Solve the nonlinear projected version of \eqref{phi:equation},
\begin{align}
	\label{phi:equation:prjct}
	\begin{cases}
		{\mathscr{L}_{\mu,{\xi}}\phi\qty[\mu,\xi]}=N\qty(\phi\qty[\mu,\xi];\mu,\xi)+\eps k\qty(U_{\mu,\xi}+\phi\qty[\mu,\xi])_{+}^{\frac{N+2}{N-2}}-\sum_{j=0}^N c_{\eps,j}\qty(\mu,\xi) H_{j;\xi}^{\mu},\\
		\int_{\R^N}\phi\qty[\mu,\xi]H_{j;\xi}^{\mu}=0,\qq{}
		j=0,1,\ldots,N, 
	\end{cases}	
\end{align} 
where
\begin{align*}
	c_{\eps,j}\qty(\mu,\xi)=\frac{1}{\displaystyle\int_{\R^N}H_{j}Z_{j}}
	\int_{\R^N}\qty(N\qty(\phi\qty[\mu,\xi];\mu,\xi)+\eps k\qty(U_{\mu,\xi}+\phi)_{+}^{\frac{N+2}{N-2}})Z_{j;\mu,\xi}.
\end{align*}
In order to do this, in \Cref{sec:lt}, we firstly develop a theory of
solvability to associated linear problem
\begin{align*}
\begin{dcases}
	{\mathscr{L_{\mu,\xi}}\phi} = g-
	\sum_{j=0}^{N}
	\frac{\displaystyle\int_{\R^N}gZ_{j;\mu,\xi}}{\displaystyle\int_{\R^N}H_jZ_j}H_{j;\xi}^{\mu},&
	\\
	\int_{\R^N}{\phi}{H_{j;\xi}^{\mu}}=0,&\qq{}j=0,1,\ldots,N,
\end{dcases}				
\end{align*}
where $g\in Y_{\mu,\xi}$. Then in \Cref{sec:nlprjct}, by making use of the Banach contraction principle, we show that
there exists one unique solution $\phi\qty[\mu,\xi]$ to \eqref{phi:equation:prjct} in $X_{\mu,\xi}$. Moreover, we prove that the solution $\phi\qty[\mu,\xi]$ is $C^1$ with respect to
$\qty(\mu,\xi)$. 
\subsubsection*{Step 2}
By making use of a topological index argument,  
in \Cref{sec:fntdim}, we will show that, for each $\eps$ sufficiently small, there exists $\qty(\mu^{\eps},\xi^{\eps})\in\qty(0,+\infty)\times\R^N$ such that
\begin{align*}
	c_{\eps,j}\qty(\mu^{\eps},\xi^{\eps})=0,\qq{} j=0,1,\ldots,N,
\end{align*}
which enables us to finish the proof of \Cref{thm:main} in \Cref{sec:pf}.

\section{Preliminary results}
In this section, we gather several preliminary results.

First of all, we recall the following proposition which plays an important role in the analysis of linear theory.
\begin{prop}\cite{TX2022}
\label{prop:nondeg}
If $\Phi$ is bounded with $\lim_{\abs{x}\to+\infty}\Phi\qty(x)=0$ satisfies
\begin{align*}
	-{\Laplace \Phi}
	=\alpha p
	\Riesz{U^{p-1}Z_{j}}{U^{p-1}}
	+\alpha\qty(p-1)\Riesz{U^{p}}{U^{p-2}}\Phi,
\end{align*}
then $\Phi$ is a  linear combination of $Z_0$,~$Z_1$,~$\ldots$,~$Z_N$.
\end{prop} 

%
%
%
%
%
%
%

Next, we recall some classical results about the Poisson equations.
\begin{lemm}
\label{lem:possion}
For any $g\in Y$, there exists a unique function $\phi\in X$
satisfying
\begin{equation}
	\label{possion}
	-\Laplace\phi = g,\qq{in}\R^N.
\end{equation} 
Moreover, we have
\begin{equation}
	\label{g2phi}
	{\phi}\qty(x)=\frac{1}{N\qty(N-2)\omega_N}\int_{\R^N}
	\frac{\fct{g}{y}}{\abs{x-y}^{N-2}}\odif{y},
	\qq{with}
	\omega_N=\frac{2\pi^{\frac{N}{2}}}{N\fct{\Gamma}{\frac{N}{2}}},
\end{equation}
and
\begin{align}
	\label{est:possion}
	\norm{\phi}_X\leq C\norm{g}_Y.
\end{align}	
\end{lemm}
\begin{proof}
Let $g\in Y$ be arbitrarily fixed. 
Using the fact that
\begin{align*}
	\int_{\R^N}\abs{g\qty(x)}^{\frac{2N}{N+2}}\odif{x}
	\leq \norm{g}_{Y}^{\frac{2N}{N+2}}\int_{\R^N}\frac{1}{\jp{x}^{2N}}
	\odif{x}\leq C \norm{g}_{Y}^{\frac{2N}{N+2}},
\end{align*}
we have $g\in L^{\frac{2N}{N+2}}$, and therefore 
there exists a unique function $\phi\in\dot{H}^1$ satisfying \eqref{possion}. Moreover, notice that $g\in C\qty(\R^N)$, we have $\phi\in C^2\qty(\R^N)$. By \cite[Theorem 6.21]{LL2001}, the function $\phi$ satisfies the integral equation \eqref{g2phi}.

Next, by making use of the inequality
$\abs{g\qty(x)}\leq \frac{1}{\jp{x}^{N+2}}\norm{g}_Y$,
we deduce from \eqref{g2phi} that,
\begin{equation*}
	\abs{{\phi}\qty(x)}
	\leq 
	\frac{\norm{g}_Y}{N\qty(N-2)\omega_N}\int_{\R^N}
	\frac{1}{\abs{x-y}^{N-2}}\frac{1}{\jp{y}^{N+2}}
	\odif{y},
\end{equation*}
which, together with \Cref{lem:identities},
implies that 
\begin{align*}
	\jp{x}^{N-2}\abs{{\phi}\qty(x)}\leq C\norm{g}_{Y},
\end{align*}
Therefore, \eqref{est:possion} holds. This ends the proof of \Cref{lem:possion}.	
\end{proof}

\begin{coro}
\label{coro:possion}
Suppose $g\in Y$ satisfies $\int_{\R^N}{g}{Z_j}=0$, $j=0,1,\ldots,N$. 
Then the function $\phi\in X$ given by \eqref{g2phi} is the unique solution to the problem
\begin{align}
\label{possion:otgn}
\begin{dcases}
	-\Laplace\phi = g,&
	\\
	\phi\in X,\qq{and} \int_{\R^N}{\phi}{H_j}=0,&\quad j=0,1,\ldots,N.
\end{dcases}		
\end{align}	
\end{coro}
\begin{proof}
Let $g\in Y$ satisfy $\int_{\R^N}{g}{Z_j}=0$, $j=0,1,\ldots,N$.
By \Cref{lem:possion}, we only need to verify that
\begin{align}
	\label{phi:otgn:H}
	\int_{\R^N}{\phi}{H_j}=0,&\qq{}j=0,1,\ldots,N.
\end{align}

By \eqref{g2phi}, we obtain
\begin{align*}
	\int_{\R^N}{\phi}{H_j}
	=
	\int_{\R^N}\qty(\frac{1}{N\qty(N-2)\omega_N}\int_{\R^N}
	\frac{\fct{g}{y}}{\abs{x-y}^{N-2}}\odif{y})H_j\qty(x)\odif{x}.
\end{align*}
By making use of the Fubini theorem, we get
\begin{align*}
	\int_{\R^N}{\phi}{H_j}
	=
	\int_{\R^N}\qty(\frac{1}{N\qty(N-2)\omega_N}\int_{\R^N}
	\frac{H_j\qty(x)}{\abs{y-x}^{N-2}}\odif{x})g\qty(y)\odif{y},
\end{align*}
which, toghether with \eqref{H2Z}, implies that
\begin{align*}
	\int_{\R^N}{\phi}{H_j}=	\int_{\R^N}{g}{Z_j}.
\end{align*}
Since $\int_{\R^N}{g}{Z_j}=0$, $j=0,1,\ldots,N$, we have \eqref{phi:otgn:H} holds. This ends the proof of \Cref{coro:possion}.
\end{proof}
%
%
\section{Linear theory}
\label{sec:lt}
In this section, we mainly 
develop the theory of solvability to the linear equation \begin{align}
\label{lin:nhmgns:nothgn:prmt}
\begin{dcases}
	{\mathscr{L_{\mu,\xi}}\phi} = g-
	\sum_{j=0}^{N}
	\frac{\displaystyle\int_{\R^N}gZ_{j;\mu,\xi}}{\displaystyle\int_{\R^N}H_jZ_j}H_{j;\xi}^{\mu},
	\\
	\int_{\R^N}{\phi}{H_{j;\xi}^{\mu}}=0,\qq{}j=0,1,\ldots,N,
\end{dcases}				
\end{align}
where
\begin{equation}
\label{L:prmt}
{\mathscr{L}_{\mu,\xi}\phi}
=
-{\Laplace\phi}
-\alpha p\Riesz{U_{\mu,\xi}^{p-1}\phi}{U_{\mu,\xi}^{p-1}}
-\alpha\qty(p-1)\Riesz{U_{\mu,\xi}^{p}}{U_{\mu,\xi}^{p-2}}{\phi}.			
\end{equation}
The rest of this section is devoted to the proof of the following proposition.
\begin{prop}
\label{prop:linsltn:prmt}
For any $\qty(\mu,\xi)\in\qty(0,+\infty)\times\R^N$, and  any $g \in Y_{\mu,\xi}$ with
\begin{equation}
\label{g:othgn:prmt}
\int_{\R^N}{g}{Z_{j;\xi}^{\mu}}=0,\qq{} j=0,1,\ldots,N,	
\end{equation}
there exists a unique function $\phi=\T\qty[{\mu,\xi}]g$ in $X_{\mu,\xi}$ solving \eqref{lin:nhmgns:nothgn:prmt}.
Moreover, we have the following estimate
\begin{align}
\label{g2phi:prmt:est}
	\norm{\phi}_{X_{\mu,\xi}}\leq C_0 \norm{g}_{Y_{\mu,\xi}},
\end{align}
where $C_0$ is a positive constant (independent of $\mu$ and $\xi$).
%
\end{prop}
In order to prove \Cref{prop:linsltn:prmt}, we require
the following lemma.

\begin{lemm}
\label{lem:linsltn}
For any $g\in Y$ with
\begin{equation}
\label{g:othgn}
	\int_{\R^N}{g}{Z_j}=0,\qq{} j=0,1,\ldots,N,
\end{equation}
there exists a unique $\phi\in X$ which solves 
\begin{align}
\label{lin:nonhmgns}
\begin{dcases}
	\Lscr\phi = g,
	\\
	\int_{\R^N}{\phi}{H_j}=0,\quad j=0,1,\ldots,N,
\end{dcases}
\end{align}
where
\begin{align}
\label{L}
{\mathscr{L}\phi}
=
-{\Laplace\phi}
-\alpha p\Riesz{U^{p-1}\phi}{U^{p-1}}
-\alpha\qty(p-1)\Riesz{U^{p}}{U^{p-2}}{\phi}.
\end{align}		
Moreover, there exists a positive constant $C$ such that,
\begin{equation}
\label{g2phi:bd}
	\norm{{\phi}}_X\leq C \norm{g}_{Y}.
\end{equation}	
\end{lemm}
\begin{proof}
First of all, let us define the linear subspace of $X$,
\begin{equation*}
	\D:=\setc*{\varphi\in X}{\int_{\R^N}{\varphi}{H_j}=0,\quad j=0,1,\ldots,N}.
\end{equation*}


\textbf{Step 1.} We claim that, for any $\phi\in\D$, 
\begin{equation}
\label{temp:05}
\alpha p\Riesz{U^{p-1}\phi}{U^{p-1}}
	+\alpha\qty(p-1)\Riesz{U^{p}}{U^{p-2}}{\phi}\in Y,
\end{equation}
and
\begin{align}
\label{temp:06}
\int_{\R^N}\qty(\alpha p\Riesz{U^{p-1}\phi}{U^{p-1}}
+\alpha\qty(p-1)\Riesz{U^{p}}{U^{p-2}}{\phi})Z_j =0,	
\end{align}
where $j=0,1,\ldots,N$.

Indeed, for any $\phi\in \D$, we have
$
	\abs{\phi\qty(x)}\leq \frac{1}{\jp{x}^{N-2}}\norm{\phi}_X
$,
 and therefore,
\begin{multline}
\label{temp:3}
\abs{\Riesz{U^{p-1}\phi}\qty(x){U^{p-1}}\qty(x)}
+
\abs{\Riesz{U^{p}}\qty(x){U^{p-2}\qty(x)}{\phi}\qty(x)}
\\
\leq 
\norm{\phi}_X\int_{\R^N}
\frac{1}{\abs{x-y}^{\lambda}}\frac{1}{\jp{y}^{2N-\lambda}}\odif{y}
\frac{1}{\jp{x}^{N+2-\lambda}}.
\end{multline}
Using \Cref{lem:decay}, we derive from \eqref{temp:3} that
\begin{align*}
\abs{\Riesz{U^{p-1}\phi}\qty(x){U^{p-1}}\qty(x)}
+
\abs{\Riesz{U^{p}}\qty(x){U^{p-2}\qty(x)}{\phi}\qty(x)}
\leq C \norm{\phi}_X\frac{1}{\jp{x}^{N+2}},
\end{align*}
which implies that \eqref{temp:05} holds.
	
Now, by the Fubini theorem, we have 
\begin{align*}
	&\int_{\R^N}
	\qty(\alpha p\Riesz{U^{p-1}\phi}{U^{p-1}}
	+\alpha\qty(p-1)\Riesz{U^{p}}{U^{p-2}}{\phi})Z_j
	\\
	=&
	\int_{\R^N}
	\qty(\alpha p\Riesz{U^{p-1}Z_j}{U^{p-1}}
	+\alpha\qty(p-1)\Riesz{U^{p}}{U^{p-2}}{Z_j})\phi,
\end{align*}
which, together with the fact that ${\mathscr{L}Z_j}=0$, implies that,
\begin{align}
	\label{temp:2}
	\int_{\R^N}
	\qty(\alpha p\Riesz{U^{p-1}\phi}{U^{p-1}}
	+\alpha\qty(p-1)\Riesz{U^{p}}{U^{p-2}}{\phi})Z_j
	=
	-\int_{\R^N}\Laplace{Z_j}\phi.
\end{align}
By making use of \eqref{H2Z}, we obtain, from \eqref{temp:2}, that
\begin{align*}
	\int_{\R^N}
	\qty(\alpha p\Riesz{U^{p-1}\phi}{U^{p-1}}
	+\alpha\qty(p-1)\Riesz{U^{p}}{U^{p-2}}{\phi})Z_j
	=
	N\qty(N+2)\int_{\R^N}\phi H_j,
\end{align*}
which, together with $\phi\in\D$, yields that 
\eqref{temp:06} holds.

\textbf{Step 2.} Since $g\in Y$ satisfies \eqref{g:othgn}, 
by \eqref{temp:05} and \eqref{temp:06}, using \Cref{coro:possion}, we obtain that $\phi$ solves the problem \eqref{lin:nonhmgns} if and only if 
$\phi\in\D$ satisfies
\begin{align}
\label{temp:10}
\phi\qty(x)
=\qty(\A\phi)\qty(x)+\qty(\N g)\qty(x),
\end{align} 
where $\A:\D\mapsto\D$ is given by
\begin{align}
\label{A}
\notag
\qty(\A\phi)\qty(x)
=
&
\frac{\alpha p}{N\qty(N-2)\omega_N}\int_{\R^N}
\frac{U^{p-1}\qty(y)}{\abs{x-y}^{N-2}}
\qty(
\int_{\R^N}\frac{U^{p-1}\qty(z)\phi\qty(z)}{\abs{y-z}^{\lambda}}
\odif{z}
)\odif{y}
\\
&+
\frac{\alpha \qty(p-1)}{N\qty(N-2)\omega_N}\int_{\R^N}
\frac{U^{p-2}\qty(y)\phi\qty(y)}{\abs{x-y}^{N-2}}
\qty(
\int_{\R^N}\frac{U^{p}\qty(z)}{\abs{y-z}^{\lambda}}
\odif{z}
)\odif{y},
\end{align}
and $\N:\setc{g\in Y}{\int_{\R^N} g Z_j=0,\quad j=0,1,\ldots,N}\mapsto \D$ is given by
\begin{align}
\label{Newton}
\qty(\N g)\qty(x)
=\frac{1}{N\qty(N-2)\omega_N}\int_{\R^N}	
\frac{g\qty(y)}{\abs{x-y}^{N-2}}\odif{y}.
\end{align}

\textbf{Step 3.} We claim that the operator $\A:\D\mapsto\D$ is compact.

Let $\{\phi_n\}_{n=1}^{\infty}$ be a bounded sequence in $\D$, i.e.
there exists a positive constant $C$ such that
\begin{equation}
\label{temp:1}
\norm{{\phi_n}}_X\leq C,\qq{for any} n\in\mathbb{N}.
\end{equation}

Notice that, by \eqref{A}, we have
\begin{align*}
-{\Laplace\qty({\A\phi_n})}
=\alpha p\Riesz{U^{p-1}\phi_n}{U^{p-1}}
+\alpha\qty(p-1)\Riesz{U^{p}}{U^{p-2}}{\phi_n},	
\end{align*}
which, together with the fact that $\phi_n\in C\qty(\R^N)$,  implies that ${\A\phi_n}\in C^1\qty(\R^N)$, and
\begin{align}
\label{C1bd}
\abs{\fct{\grad\qty({\A\phi_n})}{x}}
\leq &
C\norm{{\phi_n}}_X,	
\end{align}
where we used \Cref{lem:decay} and the fact that $\abs{\phi_n\qty(x)}\leq\frac{\norm{\phi_n}_X}{\jp{x}^{N-2}}$
in the above inequality.
Moreover, by making use of a similar argument as that in Step 1, we obtain that
\begin{align}
\label{temp:07}
	\norm{\A\phi_n}_{X}\leq C\norm{\phi_n}_X.
\end{align}

Then, by \eqref{C1bd} and \eqref{temp:07}, we have, 
\begin{equation}
	\norm{{\A\phi_n}}_{\infty}\leq C,
	~~\text{and}~~
\abs{ \A\phi_n\qty(x_2)-\A\phi_n\qty(x_1) }
\leq C \abs{ x_2-x_1 },~~\text{for any}~~x_1,x_2\in\R^N.	
\end{equation}
Hence, by the Arzel\'{a}–Ascoli theorem, there exist
a subsequence (still denoted by $\phi_n$) and a function
$\phi_0\in C\qty(\R^N)$ such that
\begin{align}
\label{temp:4}
\A\phi_n\qty(x)\rightrightarrows \phi_0\qty(x),
\qq{ uniformly on any compact subset of} \R^N.
\end{align}

By \eqref{temp:07}, we deduce that
\begin{align}
\label{temp:6}
\abs{\A\phi_n\qty(x)}\leq \frac{C}{\jp{x}^{N-2}},
\qq{for any} x\in\R^N,
\end{align}
which, toghether with \eqref{temp:4}, implies that
\begin{align}
\label{temp:7}
\abs{\phi_0\qty(x)}\leq \frac{C}{\jp{x}^{N-2}},
\qq{for any} x\in\R^N.
\end{align}

Next, let $R>0$ be an arbitrary number. By \eqref{temp:6} and \eqref{temp:7}, we have
\begin{align}
\label{temp:8}
\norm{{\A\phi_n}-\phi_0}_{X}
\leq 
\max_{\overline{B}\qty(0,R)}\abs{{\A\phi_n}\qty(x)-\phi_0\qty(x)}
+
\frac{C}{\jp{R}^{N-2}}.
\end{align}
By making use of \eqref{temp:4}, from \eqref{temp:8}, we deduce that
\begin{align}
\lim_{n\to\infty}\norm{{\A\phi_n}-\phi_0}_{X}
\leq \frac{C}{\jp{R}^{N-2}},
\end{align}
which, by letting $R\to +\infty$, yields $\lim_{n\to\infty}\norm{{\A\phi_n}-\phi_0}_{X}=0$. 
Therefore, $\A$ is compact.

\textbf{Step 4.} We claim that, if $\phi\in \D$ satisfies $\phi=\A\phi$, then $\phi\qty(x)\equiv 0$.

Indeed, for any $\phi\in \D$ with $\phi=\A\phi$, we
have
\begin{align}
\label{temp:9}
	-{\Laplace\phi}
	=\alpha p\Riesz{U^{p-1}\phi}{U^{p-1}}
	+\alpha\qty(p-1)\Riesz{U^{p}}{U^{p-2}}{\phi}.	
\end{align}
However, by \Cref{prop:nondeg}, there exist constants $c_0,c_1,\ldots,c_N$, such that
\begin{align}
	\phi =\sum_{j=0}^N c_j Z_j,
\end{align}
which, toghether with $\phi\in\D$, implies that $\phi\qty(x)\equiv 0$.

\textbf{Step 5.} By the Fredholm alternative theorem, we get, 
for any $g\in Y$ with $\int_{\R^N} g Z_j=0~\qty(j=0,1,\ldots,N)$,
there exists a unique $\phi$ such that \eqref{temp:10} holds. Moreover, we have
\begin{align}
	\norm{\phi}_X\leq C\norm{\N g}_{X},
\end{align}
which, toghether with \eqref{est:possion}, implise that \eqref{g2phi:bd} holds. This ends the proof of \Cref{lem:linsltn}.
\end{proof}	

\begin{coro}
\label{coro:linsltn}
For any $g\in Y$, there exists a unique $\phi\in X$ solves,
\begin{align}
	\label{lin:nonhmgns:nonothgn}
	\begin{dcases}
		\Lscr\phi = g-\sum_{j=0}^{N}\dfrac{\displaystyle \int_{\R^N} g Z_j }{\displaystyle\int_{\R^N} H_j Z_j } H_j,&
		\\
		\int_{\R^N}{\phi}{H_j}=0,&\qq{where}j=0,1,\ldots,N.
	\end{dcases}		
\end{align}
Moreover, there exists a constant $C>0$ such that,
\begin{equation}
	\label{g2phi:bd:nonothgn}
	\norm{{\phi}}_X\leq C \norm{g}_{Y}.
\end{equation}		
\end{coro}
\begin{proof}
By \Cref{lem:linsltn}, we deduce that, there exists
a unique $\phi\in X$ to problem \eqref{lin:nonhmgns:nonothgn} satisfying 
\begin{align}
\label{temp:11}
\norm{\phi}_X\leq C\norm{g-\sum_{j=0}^{N}\dfrac{\displaystyle \int_{\R^N} g Z_j }{\displaystyle\int_{\R^N} H_j Z_j } H_j}_Y.
\end{align} 

Now, notice that, for any $g\in Y$, we have $\abs{g\qty(x)}\leq \frac{1}{\jp{x}^{N+2}}\norm{g}_Y$, and therefore, 
\begin{align}
\abs{ \int_{\R^N} g Z_j }
\leq 
C\norm{g}_{Y}
\int_{\R^N}\frac{1}{\jp{x}^{2N}}\odif{x}
\leq 
C\norm{g}_{Y},
\end{align}
which, together with \eqref{temp:11}, implies that \eqref{g2phi:bd:nonothgn} holds.
\end{proof}

Now, we are able to prove \Cref{prop:linsltn:prmt}.
\begin{proof}[Proof of \Cref{prop:linsltn:prmt}]
Let $\qty(\mu,\xi)\in\qty(0,+\infty)\times\R^N$, and $g\in Y_{\mu,\xi}$ be fixed. 

Recall 
$
g_{;-\frac{\xi}{\mu}}^{\frac{1}{\mu}}\qty(x)=\mu^{\frac{N+2}{2}} g\qty(\mu x+\xi)
$,
it is easy to verify that $g_{;-\frac{\xi}{\mu}}^{\frac{1}{\mu}}\in Y$.
We show that
\begin{claim}
\label{clm:1} The function $\phi\in X_{\mu,\xi}$ solves
the problem
\begin{align}
	\label{temp:12}
	\begin{dcases}
		{\mathscr{L_{\mu,\xi}}\phi} = g-
		\sum_{j=0}^{N}
		\frac{\displaystyle\int_{\R^N}gZ_{j;\mu,\xi}}{\displaystyle\int_{\R^N}H_jZ_j}H_{j;\xi}^{\mu},
		\\
		\int_{\R^N}{\phi}{H_{j;\xi}^{\mu}}=0,\quad j=0,1,\ldots,N,
	\end{dcases}				
\end{align}
if and only if
the function $\phi_{;\frac{1}{\mu},-\frac{\xi}{\mu}}\in X$ solves 
the problem
\begin{align}
	\label{temp:15}
	\begin{dcases}
		\Lscr\phi_{;\frac{1}{\mu},-\frac{\xi}{\mu}} = g_{;-\frac{\xi}{\mu}}^{\frac{1}{\mu}}\qty(x)-\sum_{j=0}^{N}\dfrac{\displaystyle \int_{\R^N} g_{;-\frac{\xi}{\mu}}^{\frac{1}{\mu}}\qty(x) Z_j }{\displaystyle\int_{\R^N} H_j Z_j } H_j,
		\\
		\int_{\R^N}{\phi}{H_j}=0,\quad j=0,1,\ldots,N,
	\end{dcases}			
\end{align}	
where
\begin{align}
\label{temp:21}
\phi_{;\frac{1}{\mu},-\frac{\xi}{\mu}}\qty(x)=\mu^{\frac{N-2}{2}}\phi\qty(\mu x+\xi),
\qq{and}
g_{;-\frac{\xi}{\mu}}^{\frac{1}{\mu}}\qty(x)=\mu^{\frac{N+2}{2}} g\qty(\mu x+\xi).	
\end{align}
\end{claim}

We postpone the proof of \Cref{clm:1}. 
Since $g_{;-\frac{\xi}{\mu}}^{\frac{1}{\mu}}\in Y$, by \Cref{coro:linsltn}, there exists a unique function $\phi_{;\frac{1}{\mu},-\frac{\xi}{\mu}}\in X$ solving \eqref{temp:15} and satisfying 
\begin{align}
\label{temp:24}
	\norm{ \phi_{;\frac{1}{\mu},-\frac{\xi}{\mu}}}_X\leq
	C\norm{g_{;-\frac{\xi}{\mu}}^{\frac{1}{\mu}}}_{Y}. 
\end{align}
Therefore, by \Cref{clm:1}, the problem \eqref{temp:12} admits a unique solution $\phi\in X_{\mu,\xi}$. Moreover, it is easy to check that
\begin{align}
\label{temp:25}
	\norm{\phi}_{X_{\mu,\xi}}\leq C \norm{g}_{Y_{\mu,\xi}},
\end{align}
where the constant $C$ appeared in \eqref{temp:25} is the same as that appeared in \eqref{temp:24}. Hence \eqref{g2phi:prmt:est} holds.

The rest of the proof is due to show that \Cref{clm:1} holds.

Firstly, we easily verify that
\begin{gather}
\label{temp:16}
	-\Laplace\phi\qty(x)
	=
	-\frac{1}{\mu^{\frac{N+2}{2}}}
	\Laplace\phi_{;\frac{1}{\mu},-\frac{\xi}{\mu}}\qty(\frac{x-\xi}{\mu}).
\end{gather}

Secondly, by making a change of variables, we get
\begin{align}
\label{temp:17}
\Riesz{U_{\mu,\xi}^{p-1}\phi}\qty(x)U_{\mu,\xi}^{p-1}\qty(x)
=&
\frac{1}{\mu^{\frac{N+2}{2}}}\Riesz{U^{p-1}\phi_{;\frac{1}{\mu},-\frac{\xi}{\mu}} }\qty({\frac{x-\xi}{\mu}})U^{p-1}\qty({\frac{x-\xi}{\mu}}),
\end{align}
and
\begin{align}
\label{temp:18}
\Riesz{U_{\mu,\xi}^p}\qty(x)U_{\mu,\xi}^{p-2}\qty(x)\phi\qty(x)
=&
\frac{1}{\mu^{\frac{N+2}{2}}}\Riesz{U^{p}}\qty({\frac{x-\xi}{\mu}})
U^{p-2}\qty({\frac{x-\xi}{\mu}})\phi_{;\frac{1}{\mu},-\frac{\xi}{\mu}}\qty({\frac{x-\xi}{\mu}}).	
\end{align}

Finally, straightforward computations show that 
\begin{align}
\label{temp:19}
	\int_{\R^N}g Z_{j;\mu,\xi}
	=
	\int_{\R^N}g_{;-\frac{\xi}{\mu}}^{\frac{1}{\mu}}Z_j,
	\qq{} j=0,1,\ldots,N,
\end{align}
and
\begin{align}
\label{temp:20}
	\int_{\R^N}\phi H_{j;\xi}^{\mu}
	=
	\int_{\R^N}\phi_{;\frac{1}{\mu},-\frac{\xi}{\mu}}H_j,
	\qq{} j=0,1,\ldots,N.
\end{align}

On the one hand, by combining \eqref{temp:16} with \eqref{temp:17} and \eqref{temp:18},
we get
\begin{align}
\label{temp:22}
	{\mathscr{L_{\mu,\xi}}\phi}\qty(x)
	=
	\frac{1}{\mu^{\frac{N+2}{2}}}{\mathscr{L}\phi_{;\frac{1}{\mu},-\frac{\xi}{\mu}}}\qty({\frac{x-\xi}{\mu}}).
\end{align}
On the other hand, we deduce, from \eqref{temp:19}, that
\begin{multline}
\label{temp:23}
g\qty(x)-
\sum_{j=0}^{N}
\frac{\displaystyle\int_{\R^N}gZ_{j;\mu,\xi}}{\displaystyle\int_{\R^N}H_jZ_j}H_{j;\xi}^{\mu}\qty(x)
\\
=
\frac{1}{\mu^{\frac{N+2}{2}}}
\qty(
g_{;-\frac{\xi}{\mu}}^{\frac{1}{\mu}}\qty({\frac{x-\xi}{\mu}})
-
\sum_{j=0}^{N}
\frac{\displaystyle\int_{\R^N}g_{;-\frac{\xi}{\mu}}^{\frac{1}{\mu}} Z_{j}}{\displaystyle\int_{\R^N}H_jZ_j}H_{j}
\qty({\frac{x-\xi}{\mu}})
).
\end{multline}
By combining \eqref{temp:22} and \eqref{temp:23} with \eqref{temp:20}, we conclude that
$\phi$ solves \eqref{temp:12} if and only if $\phi_{;\frac{1}{\mu},-\frac{\xi}{\mu}}$ solves \eqref{temp:15}. This ends the proof of \Cref{clm:1}. 

Hence, we complete the proof of \Cref{prop:linsltn:prmt}.
\end{proof}

We end this section with analysing the differentiability properties of the linear operator $\T_{\mu,\xi}$ with respect to $\qty(\mu,\xi)$, where $\T_{\mu,\xi}$ is given by \Cref{prop:linsltn:prmt}. More precisely, we have:
\begin{lemm}
\label{lemm:C1}
Let $\T\qty[{\mu,\xi}]:Y_{\mu,\xi}\mapsto X_{\mu,\xi}$ be given by \Cref{prop:linsltn:prmt}. Then the map $\qty(\mu,\xi)\mapsto \T\qty[{\mu,\xi}]$ is continuously differentiable. Moreover, for any $g\in Y_{\mu,\xi}$, we have
\begin{equation}
	\norm{\frac{\partial\T_{\mu,\xi}}{\partial\mu}g}_{X_{\mu,\xi}}
	+
\sum_{j=1}^N \norm{\frac{\partial\T_{\mu,\xi}}{\partial\xi_j}g}_{X_{\mu,\xi}}
	\leq \frac{C}{\mu}
	\norm{g}_{Y_{\mu,\xi}}.	
\end{equation}	
\end{lemm}
Since the proof of \Cref{lemm:C1} follows from tedious calculations, we postpone it to \Cref{app:C1}.

\section{Nonlinear projected problem}
\label{sec:nlprjct}
In this section, via the linear theory developed in \Cref{prop:linsltn:prmt}, by making use of the Banach contraction mapping theory, 
we solve the non-linear problem
\begin{align}
\label{NLP}
\begin{dcases}
{\mathscr{L}_{\mu,{\xi}}\phi } = N\qty(\phi;\mu,\xi)
+\eps E\qty(\phi;\mu,\xi)-\sum_{j=0}^{N}c_{\eps,j}H_{j;\xi}^{\mu},&
	\\
\int_{\R^N}{\phi }H_{j;\mu}^{\xi}=0,\qq{ }j=0,1,\ldots,N,
\end{dcases}				
\end{align}
where $\mathscr{L}_{\mu,{\xi}}$ is given by \eqref{L:prmt},
\begin{align}
\label{N}
\notag
N\qty(\phi;\mu,\xi)
=&
{\alpha}\Riesz{\qty(U_{\mu,\xi}+\phi)_{+}^p}\qty(U_{\mu,\xi}+\phi)_{+}^{p-1}
-
{\alpha}\Riesz{U_{\mu,\xi}^p}U_{\mu,\xi}^{p-1}
\\
&
-
{\alpha}p\Riesz{U_{\mu,\xi}^{p-1}\phi}U_{\mu,\xi}^{p-1}
-
{\alpha}\qty(p-1)\Riesz{U_{\mu,\xi}^{p}}U_{\mu,\xi}^{p-2}\phi,
\end{align}
\begin{align}
	\label{E}
	E\qty(\phi;\mu,\xi) = k\qty(U_{\mu,\xi}+\phi)_{+}^{\frac{N+2}{N-2}},
\end{align}
and
\begin{align}
\label{cj}
	c_{\eps,j}\qty(\mu,\xi)=\frac{1}{\displaystyle\int_{\R^N}H_{j}Z_{j}}
	\int_{\R^N}\qty(N\qty(\phi\qty[\mu,\xi];\mu,\xi)+\eps E\qty(\phi;\mu,\xi))Z_{j;\mu,\xi}.
\end{align}

The main result of this section is the following proposition.
\begin{prop}
\label{prop:nlp}
Let $\qty(\mu,\xi)\in\qty(0,+\infty)\times\R^N$, and the function $k$ satisfy \ref{k:pstv}-\ref{k:reg}. For any $\eps>0$ sufficiently small, the problem \eqref{NLP} admits a unique solution $\bm{\phi} = \bm{\phi}\qty[\mu,\xi]\in X_{\mu,\xi}$ satisfying 
\begin{align}
\label{phi:bd}
	\norm{\bm{\phi}\qty[\mu,\xi]}_{X_{\mu,\xi}}\leq C\norm{k}_{\infty}\eps.
\end{align}
Moreover, the function $\bm{\phi}\qty[\mu,\xi]$ is $C^1$ with respect to $\qty(\mu,\xi)$, and 
\begin{align}
\label{pd:phi:bd}
\norm{\pdv{\bm{\phi}}{\mu}\qty[\mu,\xi]}_{X_{\mu,\xi}}
+
\sum_{j=1}^N\norm{\pdv{\bm{\phi}}{\xi_j}\qty[\mu,\xi]}_{X_{\mu,\xi}}
\leq \frac{C}{\mu}\norm{k}_{\infty}{\eps}.
\end{align} 
\end{prop}

\begin{proof}
By \Cref{prop:linsltn:prmt}, we deduce that $\phi$ solves \eqref{NLP} if and only if
\begin{align}
\label{IF:NLP}
\phi 
=
\K\qty[\mu,\xi]\qty(\phi)
,\qq{} \phi\in X_{\mu,\xi},
\end{align}
where
\begin{align}
\label{K}
	\K\qty[\mu,\xi]\qty(\phi)
	=
	\T\qty[\mu,\xi]\qty(
	N\qty(\phi;\mu,\xi)
	+\eps E\qty(\phi;\mu,\xi)
	).
\end{align}

Now, for each $\rho\in\qty(0,\frac{1}{2})$, we introduce the subset of $X_{\mu,\xi}$:
\begin{align*}
	B_{\rho}=\setc{\varphi\in X_{\mu,\xi}}{ \norm{\varphi}_{X_{\mu,\xi}}< \rho }.
\end{align*}

\subsection*{Step 1}
\label{set5:step1}
 We show that, for any $\eps>0$ sufficiently small, there exists $\rho_0\in\qty(0,\frac{1}{2})$ sufficiently small, such that, for each $\qty(\mu,\xi)\in\qty(0,+\infty)\times\R^N$, 
there exists a unique $\bm{\phi}\qty[\mu,\xi]$ such that
\begin{align}
\label{sec5:temp21}
	\bm{\phi}\qty[\mu,\xi]=\K\qty[\mu,\xi]\qty(\bm{\phi}\qty[\mu,\xi]),
\qq{with }\bm{\phi}\qty[\mu,\xi]\in \bar{B}_{\rho_0}.
\end{align}

First, notice that, for any $\phi\in \bar{B}_{\rho}$,  we have 
$\abs{\phi\qty(x)}\leq\frac{1}{2}U_{\mu,\xi}\qty(x)$, and therefore, 
\begin{align}
	\label{sec5:temp01}
	\abs{U_{\mu,\xi}\qty(x)+\phi\qty(x)}\leq \frac{3}{2}U_{\mu,\xi}\qty(x),
\end{align}
which, together with \eqref{N} and \Cref{lem:decay}, implies that
\begin{gather}
\label{sec5:temp20}
	\abs{ N\qty(\phi;\mu,\xi)\qty(x) }\leq C_1 U_{\mu,\xi}^{\frac{6-N}{N-2}}\qty(x)\abs{\phi\qty(x)}^2,
\end{gather}
where the positive constant $C_1$ is independent of $\mu$ and $\xi$.
Hence, we have
\begin{align}
	\label{sec5:temp02}
	\norm{N\qty(\phi;\mu,\xi)}_{Y_{\mu,\xi}}
	\leq C_1 \rho^2.	
\end{align}
Moreover, using \eqref{sec5:temp01} and \Cref{lem:decay}, we get
\begin{align}
	\label{sec5:temp03}
	\norm{E\qty(\phi;\mu,\xi)}_{Y_{\mu,\xi}}
	\leq C_2 \norm{k}_{\infty}.
\end{align}
where the positive constant $C_2$ is independent of $\mu$ and $\xi$.
By \Cref{prop:linsltn:prmt}, we deduce, from \eqref{sec5:temp02} and \eqref{sec5:temp03}, that
\begin{align}
\label{sec5:temp06}
	\norm{\K\qty[\mu,\xi]\qty(\phi)}_{X_{\mu,\xi}}
	\leq C_0\qty(C_1\rho^2+\eps C_2\norm{k_{\infty}}),
	\qq{for any} \phi\in \bar{B}_{\rho}.
\end{align}

Next, notice that, for any $\phi\in \bar{B}_{\rho}$ and $\tilde{\phi}\in \bar{B}_{\rho}$, we have
\begin{align*}
	\abs{N\qty(\phi;\mu,\xi)\qty(x)-N\qty(\tilde{\phi};\mu,\xi)\qty(x)}
	\leq C_3 
	U_{\mu,\xi}^{\frac{6-N}{N-2}}\qty(x)
	\qty(\abs{\phi\qty(x)}+\abs{\tilde{\phi}\qty(x)})
	\abs{{\phi\qty(x)}-{\tilde{\phi}\qty(x)}},
\end{align*}
therefore,
\begin{align}
\label{sec5:temp04}
	\norm{N\qty(\phi;\mu,\xi)-N\qty(\tilde{\phi};\mu,\xi)}_{Y_{\mu,\xi}}
	\leq  
	2 C_3\rho \norm{{\phi}-{\tilde{\phi}}}_{X_{\mu,\xi}}.
\end{align}
where the positive constant $C_3$ is independent of $\mu$ and $\xi$.
Moreover, by direct computations, we get
\begin{align}
\label{sec5:temp05}
	\norm{E\qty(\phi;\mu,\xi)-E\qty({\tilde{\phi}};\mu,\xi)}_{Y_{\mu,\xi}}
	\leq C_4\norm{k}_{\infty} \norm{{\phi}-{\tilde{\phi}}}_{X_{\mu,\xi}}.
\end{align}
where the positive constant $C_4$ is independent of $\mu$ and $\xi$.
Hence, by making use of \Cref{prop:linsltn:prmt}, \eqref{sec5:temp04} and \eqref{sec5:temp05}, we get,
for any $\phi,~\tilde{\phi}\in \bar{B}_{\rho}$, 
\begin{align}
\label{sec5:temp07}
\norm{\K\qty[\mu,\xi]\qty(\phi)-\K\qty[\mu,\xi]\qty(\tilde{\phi})}_{X_{\mu,\xi}}
\leq C_0\qty(2C_3\rho +\eps C_4\norm{k_{\infty}} )\norm{{\phi}-{\tilde{\phi}}}_{X_{\mu,\xi}}.
\end{align}

Finally, using the fact that $\eps>0$ is sufficiently small, by taking
\begin{align}
\label{sec5:temp17}
\rho_0=\min\{2C_0 C_2\eps\norm{k_{\infty}},~\frac{C_4}{2C_3}\eps\norm{k_{\infty}} \},
\end{align} 
we get
\begin{align*}
C_0\qty(C_1\rho_0^2+\eps C_2\norm{k_{\infty}})<\rho_0,
\qq{and}
C_0\qty(2C_3\rho_0 +\eps C_4\norm{k_{\infty}})<\frac{1}{2}.	
\end{align*}
Therefore, by \eqref{sec5:temp06} and \eqref{sec5:temp07}, we have, for any $\phi$, $\tilde{\phi}\in \bar{B}_{\rho_0}$,
\begin{align*}
\norm{\K\qty[\mu,\xi]\qty(\phi)}_{X_{\mu,\xi}}
\leq \rho_0,
\qq{and}
\norm{\K\qty[\mu,\xi]\qty(\phi)-\K\qty[\mu,\xi]\qty(\tilde{\phi})}_{X_{\mu,\xi}}\leq 
\frac{1}{2}\norm{{\phi}-{\tilde{\phi}}}_{X_{\mu,\xi}},
\end{align*}
which, together with the contraction mapping theorem, implies that \eqref{sec5:temp21} holds.

\subsection*{Step 2.}We show that \eqref{pd:phi:bd} holds. 

Let $\qty(\mu_0,\xi_0)\in\qty(0,+\infty)\times\R^N$ be arbitrary fixed. 
We denote $\bm{\Phi}_0 = \bm{\phi}\qty(\mu_0,\xi_0)$
and
\begin{align*}
\A\qty( \phi; \mu,\xi )
=
\phi 
-
\K\qty[\mu,\xi]\qty(\phi).	
\end{align*}
Obviously, by \eqref{sec5:temp21}, we have
\begin{align}
\label{sec5:temp13}
	\A\qty(\bm{\Phi}_0; \mu_0,\xi_0 )=0.
\end{align}

By direct computations, we get
\begin{align}
\label{sec5:temp10}
\pdv{\A}{\phi}\qty( \bm{\Phi}_0; \mu_0,\xi_0 )\psi
=\psi
-
\T\qty[\mu_0,\xi_0]\qty(
\pdv{N}{\phi}\qty(\bm{\Phi}_0;\mu_0,\xi_0)\psi
+\eps \pdv{E}{\phi}\qty(\bm{\Phi}_0;\mu_0,\xi_0)\psi
),	
\end{align}
where
\begin{align}
\notag
&\pdv{N}{\phi}\qty(\bm{\Phi}_0;\mu_0,\xi_0)\psi
\\
\notag
=
&
\alpha p 
\qty(
\Riesz{\qty(U_{\mu_0,\xi_0}+\bm{\Phi}_0)^{p-1}\psi}\qty(U_{\mu_0,\xi_0}+\bm{\Phi}_0)^{p-1}
-
\Riesz{U_{\mu_0,\xi_0}^{p-1}\psi}U_{\mu_0,\xi_0}^{p-1}
)
\\
\notag
&
+
\alpha \qty(p-1) \qty(\Riesz{\qty(U_{\mu_0,\xi_0}+\bm{\Phi}_0)^{p}}\qty(U_{\mu_0,\xi_0}+\bm{\Phi}_0)^{p-2}-\Riesz{U_{\mu_0,\xi_0}^{p}}U_{\mu_0,\xi_0}^{p-2})\psi,	
\end{align}
and
\begin{align*}
\pdv{E}{\phi}\qty(\bm{\Phi}_0;\mu_0,\xi_0)\psi
=
\frac{N+2}{N-2} k \qty( U_{\mu_0,\xi_0}+\bm{\Phi}_0 )^{\frac{4}{N-2}}\psi.		
\end{align*}

By making use of a similar argument as that appeared in {Step 1}, we obtain that
\begin{align}
\label{sec5:temp08}
\norm{\pdv{N}{\phi}\qty(\bm{\Phi}_0;\mu_0,\xi_0)\psi}_{Y_{\mu_0,\xi_0}}
\leq C\norm{\bm{\Phi}_0}_{X_{\mu_0,\xi_0}}\norm{{\psi}}_{X_{\mu_0,\xi_0}},
\end{align}
and
\begin{align}
\label{sec5:temp09}
\norm{\pdv{E}{\phi}\qty(\bm{\Phi}_0;\mu_0,\xi_0)\psi}_{Y_{\mu_0,\xi_0}}
\leq
C \norm{k}_{\infty}\norm{\psi}_{X_{\mu_0,\xi_0}},	
\end{align}
where we used the fact that $\norm{\bm{\Phi}_0}_{X_{\mu_0,\xi_0}}\leq \rho_0$.
Therefore, by inserting \eqref{sec5:temp08} and \eqref{sec5:temp09} into \eqref{sec5:temp10}, we get
\begin{align}
\label{sec5:temp11}
\norm{\pdv{\A}{\phi}\qty( \bm{\Phi}_0; \mu_0,\xi_0 )\psi}_{X_{\mu_0,\xi_0}}
	>\frac{1}{2}\norm{\psi}_{X_{\mu_0,\xi_0}},	
\end{align}
where we used that $\eps$ and $\rho_0$ are sufficiently small.
It follows from \eqref{sec5:temp11} that
$\pdv{\A}{\phi}\qty( \bm{\Phi}_0; \mu_0,\xi_0 ):X_{{\mu_0,\xi_0}}\mapsto X_{{\mu_0,\xi_0}}$ is invertible with
\begin{align}
\label{sec5:temp12}
	\norm{\pdv{\A}{\phi}\qty( \bm{\Phi}_0; \mu_0,\xi_0 )^{-1}}_{L\qty( X_{\mu_0,\xi_0} , X_{\mu_0,\xi_0} )}<2.
\end{align} 
By the Implicit Function Theorem, for any $\qty(\mu,\xi)$ sufficiently close to $\qty(\mu_0,\xi_0)$, 
there exists a unique function $\bm{\Phi}\qty[\mu,\xi]$ satisfying 
\begin{align*}
	\bm{\Phi}\qty[\mu,\xi]= \K\qty[\mu,\xi]\qty(\bm{\Phi}\qty[\mu,\xi]),
\end{align*}
with $\bm{\Phi}\qty(\mu_0,\xi_0)=\bm{\Phi}_0$. Moreover, by uniqueness, we have 
\begin{align*}
	\bm{\phi}\qty[\mu,\xi]=\bm{\Phi}\qty[\mu,\xi],
\end{align*}
where $\qty(\mu,\xi)$ is sufficiently close to $\qty(\mu_0,\xi_0)$. 

Next, by direct computations, we get
\begin{align}
\notag
\label{sec5:temp16}
&\pdv{\bm{\phi}}{\mu}\qty[\mu_0,\xi_0]
\\
\notag
=&
\qty(\pdv{\A}{\phi}\qty(\bm{\phi}\qty[\mu_0,\xi_0];\mu_0,\xi_0 ))^{-1}
\pdv{\T}{\mu}\qty(\mu,\xi)
N\qty(\bm{\phi}\qty[\mu_0,\xi_0];\mu_0,\xi_0)
\\
\notag
&+\eps
\qty(\pdv{\A}{\phi}\qty(\bm{\phi}\qty[\mu_0,\xi_0];\mu_0,\xi_0 ))^{-1}
\pdv{\T}{\mu}\qty(\mu,\xi)
 E\qty(\bm{\phi}\qty[\mu_0,\xi_0];\mu_0,\xi_0)
\\
\notag
&+
\qty(\pdv{\A}{\phi}\qty(\bm{\phi}\qty[\mu_0,\xi_0];\mu_0,\xi_0 ))^{-1}
{\T}\qty(\mu_0,\xi_0)
\pdv{N}{\mu}\qty(\bm{\phi}\qty[\mu_0,\xi_0];\mu_0,\xi_0)
\\
&
+\eps
\qty(\pdv{\A}{\phi}\qty(\bm{\phi}\qty[\mu_0,\xi_0];\mu_0,\xi_0 ))^{-1}
{\T}\qty(\mu_0,\xi_0)
\pdv{E}{\mu}\qty(\bm{\phi}\qty[\mu_0,\xi_0];\mu_0,\xi_0)
,		
\end{align}
Again, by making use of a similar argument as that appeared in the proof of Step 1, we get
\begin{gather}
\label{sec5:temp14}
	\norm{N\qty(\bm{\phi}\qty[\mu_0,\xi_0];\mu_0,\xi_0)}_{Y_{\mu,\xi}}
	\leq 
	C\norm{\bm{\phi}\qty[\mu_0,\xi_0]}_{Y_{\mu_0,\xi_0}}^2,	
\\
\label{sec5:temp18}
	\norm{\pdv{N}{\mu}\qty(\bm{\phi}\qty[\mu_0,\xi_0];\mu_0,\xi_0)}_{Y_{\mu_0,\xi_0}}
	\leq 
	\frac{C}{\mu_0}\norm{\bm{\phi}\qty[\mu_0,\xi_0]}_{Y_{\mu_0,\xi_0}}^2,	
\\
\label{sec5:temp15}
	\norm{E\qty(\bm{\phi}\qty[\mu_0,\xi_0];\mu_0,\xi_0)}_{Y_{\mu_0,\xi_0}}
	\leq 
	C\norm{k}_{\infty},
\\
\label{sec5:temp19}
	\norm{\pdv{E}{\mu}\qty(\bm{\phi}\qty[\mu_0,\xi_0];\mu_0,\xi_0)}_{Y_{\mu_0,\xi_0}}
	\leq 
	\frac{C}{\mu}\norm{k}_{\infty}.
\end{gather}

Finally, using \Cref{prop:linsltn:prmt}, \Cref{lemm:C1} and \eqref{sec5:temp17}, by inserting \eqref{sec5:temp14}-\eqref{sec5:temp19} into \eqref{sec5:temp16}, we get
\begin{align*}
\norm{\pdv{\bm{\phi}}{\mu}\qty[\mu_0,\xi_0]}_{X_{\mu_0,\xi_0}}
\leq 
\frac{C}{\mu_0}\norm{k}_{\infty}\eps.	
\end{align*}
By a similar argument as above, we have
\begin{align*}
	\norm{\pdv{\bm{\phi}}{\xi_j}\qty[\mu_0,\xi_0]}_{X_{\mu_0,\xi_0}}
	\leq 
	\frac{C}{\mu_0}\norm{k}_{\infty}\eps,
\qq{} j=1,2,\ldots,N.
\end{align*}
Hence, \eqref{pd:phi:bd} holds. This ends the proof of \Cref{prop:nlp}.
\end{proof}
\section{The finite dimensional reduction}
\label{sec:fntdim}
In this section, based on the above argument, we reduce solving the full problem \eqref{NLH-eps} to finding a critical point of a finite dimensional function. 

Recall that, at least formally, the equation 
\begin{equation}
\label{NLH-eps++}
-{\Laplace u}
-{\alpha}
\Riesz{\qty(u_{+})^p}
\qty(u_{+})^{p-1}
-\eps k\qty(u_{+})^{\frac{N+2}{N-2}}
=0	
\end{equation}
is the Euler-Lagrange equation corresponds to the functional $J_{\eps}$, which is given by
\begin{align}
\label{J}
J_{\eps}\qty(u)
=
\frac{1}{2}\int_{\R^N}\abs{\grad{u}}^2
-\frac{\alpha}{2p}
\int_{\R^N}
\Riesz{u_{+}^p}u_{+}^p
-\frac{N-2}{2N}\eps\int_{\R^N}k u_{+}^{\frac{2N}{N-2}}.		
\end{align}

Now, let the function $\bm{\phi}\qty[\mu,\xi]$ be given by \Cref{prop:nlp}, we denote 
\begin{align}
\label{omega} \omega\qty[\mu,\xi]\qty(x)=U_{\mu,\xi}\qty(x)+\bm{\phi}\qty[\mu,\xi]\qty(x),			
\end{align}
and
\begin{align}
\label{jeps}
	\j_{\eps}\qty(\mu,\xi)=J_{\eps}\qty(U_{\mu,\xi}+\bm{\phi}\qty[\mu,\xi]).
\end{align}
By \eqref{phi:bd}, for $\eps$ sufficiently small, we get 
\begin{align}
	\label{sec6:07}
	\abs{\bm{\phi}\qty[\mu,\xi]\qty(x)}\leq \frac{1}{2}U_{\mu,\xi}\qty(x),\qq{for all} x\in\R^N,
\end{align}
which implies that
\begin{align*}
	\qty(\omega\qty[\mu,\xi])_{+}\qty(x)=\omega\qty[\mu,\xi]\qty(x),\qq{for all} x\in\R^N.
\end{align*}
Therefore, by \eqref{NLP}, we easily check that 
\begin{align}
\label{omega:eq}
-{\Laplace \omega\qty[\mu,\xi]}
-{\alpha}
\Riesz{\omega\qty[\mu,\xi]^p}\,
\omega^{p-1}
-\eps k\omega\qty[\mu,\xi]^{\frac{N+2}{N-2}}
=-\sum_{j=0}^N c_{\eps,j}\qty(\mu,\xi) H_{j;\xi}^{\mu},			
\end{align}
where
\begin{align}
\label{cj:6}
c_{\eps,j}\qty(\mu,\xi)=\frac{1}{\displaystyle\int_{\R^N}H_{j}Z_{j}}
\int_{\R^N}\qty(N\qty(\bm{\phi}\qty[\mu,\xi];\mu,\xi)+\eps k\omega\qty[\mu,\xi]^{\frac{N+2}{N-2}})Z_{j;\mu,\xi},		
\end{align}
with 
\begin{multline}
\label{sec6:N}
N\qty(\bm{\phi}\qty[\mu,\xi];\mu,\xi)
=
{\alpha}\Riesz{\omega\qty[\mu,\xi]^p}\omega\qty[\mu,\xi]^{p-1}
-
{\alpha}\Riesz{U_{\mu,\xi}^p}U_{\mu,\xi}^{p-1}
\\
-
{\alpha}p\Riesz{U_{\mu,\xi}^{p-1}\bm{\phi}\qty[\mu,\xi]}U_{\mu,\xi}^{p-1}
-
{\alpha}\qty(p-1)\Riesz{U_{\mu,\xi}^{p}}U_{\mu,\xi}^{p-2}\bm{\phi}\qty[\mu,\xi].	
\end{multline}

Noticing \eqref{omega:eq}, in order to obtain solutions to \eqref{NLH-eps++} 
with the form given by \eqref{omega}, we only need to find $\qty(\mu,\xi)\in \qty(0,+\infty)\times\R^N$ such that
\begin{align}
\label{sec6:01}
	c_{\eps,j}\qty(\mu,\xi)=0,\qq{}j=0,1,\ldots,N.
\end{align}

The following lemma shows that the finite dimensional problem \eqref{sec6:01} is closed related to critical points of   $\j_{\eps}$.

\begin{lemm}
\label{lem:c2J}
Let $\eps>0$ be sufficiently small and the function $k$ satisfies \ref{k:pstv}-\ref{k:reg}.
Let $\j_{\eps}$ and $c_{\eps,j}$ be given by \eqref{jeps} and \eqref{cj:6} respectively. 
Then 
\begin{align*}
c_{\eps,j}\qty(\mu,\xi)=0,\qq{}j=0,1,\ldots,N,
\end{align*}
if and only if
\begin{align*}
\nabla_{\mu,\xi}\j_{\eps}\qty(\mu,\xi)=0.
\end{align*}	
\end{lemm}
\begin{proof}
First, by direct computations, we have
\begin{align*}
&\pdv{\j_{\eps}}{\mu}\qty(\mu,\xi)
\\=&
\int_{\R^N}\qty(-{\Laplace \omega\qty[\mu,\xi]}
-{\alpha}
\Riesz{\omega\qty[\mu,\xi]^p}\,
\omega^{p-1}
-\eps k\omega\qty[\mu,\xi]^{\frac{N+2}{N-2}})
\qty( \frac{1}{\mu}Z_{0;\mu,\xi}+\pdv{\phi}{\mu}\qty[\mu,\xi] ),	
\end{align*}
which, together with \eqref{omega:eq} and \Cref{lem:identities}, implies that
\begin{align}
\label{sec6:02}
\pdv{\j_{\eps}}{\mu}\qty(\mu,\xi)
=
-c_{\eps,0}\qty(\mu,\xi)\frac{1}{\mu}\int_{\R^N}H_0Z_0
-\sum_{j=0}^{N}c_{\eps,j}\qty(\mu,\xi)\int_{\R^N} H_{j;\xi}^{\mu}\pdv{\phi}{\mu}\qty[\mu,\xi].	
\end{align}
Arguing similarly as above, we obtain that, for $i=1,\ldots,N$,
\begin{align}
\label{sec6:03}
\pdv{\j_{\eps}}{\xi_i}\qty(\mu,\xi)
=
-c_{\eps,i}\qty(\mu,\xi)\frac{1}{\mu}\int_{\R^N}H_i Z_i
-\sum_{j=0}^{N}c_{\eps,j}\qty(\mu,\xi)\int_{\R^N} H_{j;\xi}^{\mu}\pdv{\phi}{\xi_i}\qty[\mu,\xi].		
\end{align}

Now, by \eqref{sec6:02} and \eqref{sec6:03}, we easily get
\begin{align}
\label{sec6:04}
\qty(A+B_{\eps})
\begin{pmatrix}
	c_{\eps,0} \\
	c_{\eps,1} \\
	\vdots \\
	c_{\eps,N}
\end{pmatrix}
=
\begin{pmatrix}
	\pdv{\j_{\eps}}{\mu}\qty(\mu,\xi) \\
	\pdv{\j_{\eps}}{\xi_1}\qty(\mu,\xi) \\
	\vdots \\
	\pdv{\j_{\eps}}{\xi_N}\qty(\mu,\xi)
\end{pmatrix},	
\end{align}
where
\begin{align*}
A=-\frac{1}{\mu}\diag
\qty(\int_{\R^N}H_0Z_0,\int_{\R^N}H_1Z_1,\ldots,\int_{\R^N}H_N Z_N),	
\end{align*}
and
\begin{align*}
B_{\eps}=-
\begin{pmatrix}
	\int_{\R^N} H_{0;\xi}^{\mu}\pdv{\phi}{\mu}\qty[\mu,\xi] 
	& \int_{\R^N} H_{1;\xi}^{\mu}\pdv{\phi}{\mu}\qty[\mu,\xi]  
	& \dots  
	& \int_{\R^N} H_{N;\xi}^{\mu}\pdv{\phi}{\mu}\qty[\mu,\xi] 
	\\
	\int_{\R^N} H_{0;\xi}^{\mu}\pdv{\phi}{\xi_1}\qty[\mu,\xi] 
	& \int_{\R^N} H_{1;\xi}^{\mu}\pdv{\phi}{\xi_1}\qty[\mu,\xi] 
	& \dots  
	& \int_{\R^N} H_{N;\xi}^{\mu}\pdv{\phi}{\xi_1}\qty[\mu,\xi] 
	\\
	\vdots & \vdots  & \ddots & \vdots \\
	\int_{\R^N} H_{0;\xi}^{\mu}\pdv{\phi}{\xi_N}\qty[\mu,\xi] 
	& \int_{\R^N} H_{1;\xi}^{\mu}\pdv{\phi}{\xi_N}\qty[\mu,\xi] 
	& \dots  
	& \int_{\R^N} H_{N;\xi}^{\mu}\pdv{\phi}{\xi_N}\qty[\mu,\xi] 
\end{pmatrix}.	
\end{align*}

By \eqref{pd:phi:bd}, we get, for $j=0,1,\ldots,N$ and $i=1,2,\ldots,N$,
\begin{align*}
	\int_{\R^N} H_{j;\xi}^{\mu}\pdv{\phi}{\mu}\qty[\mu,\xi]  =\frac{\norm{k}_{\infty}}{\mu}\O\qty(\eps),	
\qq{and}
	\int_{\R^N} H_{j;\xi}^{\mu}\pdv{\phi}{\xi_i}\qty[\mu,\xi]  =\frac{\norm{k}_{\infty}}{\mu}\O\qty(\eps),	
\end{align*} 
which implies that
\begin{align}
\label{sec6:05}
\det(A+B_{\eps})=\frac{(-1)^{N+1}}{\mu^{N+1}}
\qty(\prod_{j=1}^{N}\int_{\R^N}H_jZ_j
+\O\qty(\eps)).
\end{align}
Therefore, by \eqref{sec6:04} and \eqref{sec6:05}, we concludes the proof of \Cref{lem:c2J}.
\end{proof}

Next, in order to study the asymptotic behaviour of $\j_{\eps}$ as $\eps\to 0+$, we introduce the following function
\begin{align}
\label{Upsilon}
\Upsilon\qty(\mu,\xi)=
\int_{\R^N}k\qty(x)U_{\mu,\xi}^{\frac{2N}{N-2}}\qty(x)\odif{x}
=
\frac{N-2}{2N}\int_{\R^N}k\qty(\mu y+\xi) U^{\frac{2N}{N-2}}\qty(y)\odif{y}.
\end{align}
For any $k$ satisfies \ref{k:pstv}, by a straight calculation, we get
\begin{gather*}
\pdv{\Upsilon}{\mu}\qty(\mu,\xi)=\frac{1}{\mu}
\int_{\R^N}k\qty(\mu y+\xi) U^{\frac{N+2}{N-2}}\qty(y)Z_0\qty(y)\odif{y},
\\
\pdv{\Upsilon}{\xi_j}\qty(\mu,\xi)=\frac{1}{\mu}
\int_{\R^N}k\qty(\mu y+\xi) U^{\frac{N+2}{N-2}}\qty(y)Z_j\qty(y)\odif{y},\qq{}j=1,2,\ldots,N.
\end{gather*}
We can prove the following: 
\begin{lemm}
\label{lem:Jexpand}
Let $\qty(\mu,\xi)\in\qty(0,+\infty)\times\R^N$. 
Suppose $k$ satisfies \ref{k:pstv}.
Let $\j_{\eps}$ and $\Upsilon$ be given by \eqref{jeps} and \eqref{Upsilon} respectively. Then, as $\eps\to 0+$, we have
\begin{gather}
\label{expand:j}
\j_{\eps}\qty(\mu,\xi)=\bar{\j}_0-\eps\Upsilon\qty(\mu,\xi)
+
\norm{k}_{\infty}^2\O\qty(\eps^2),		
\\
\label{expand:pjmu}
\pdv{\j_\eps}{\mu}\qty(\mu,\xi)
=
-\eps\pdv{\Upsilon}{\mu}\qty(\mu,\xi)
+\frac{1}{\mu}\norm{k}_{\infty}^2\O\qty(\eps^2),	
\\
\label{expand:pjxi}
\pdv{\j_\eps}{\xi_j}\qty(\mu,\xi)
=
-\eps\pdv{\Upsilon}{\xi_j}\qty(\mu,\xi)
+\frac{1}{\mu}\norm{k}_{\infty}^2\O\qty(\eps^2),
\qq{}j=1,2,\ldots,N,		
\end{gather}
where
\begin{align*}
\bar{\j}_0
=
\frac{1}{2}\int_{\R^N}\abs{\grad{U}}^2
-\frac{\alpha}{2p}
\int_{\R^N}
\Riesz{U^p}U^p.
\end{align*}
\end{lemm}
\begin{proof}
The proof follows from direct computations.
\subsection*{Proof of \eqref{expand:j}} By \eqref{jeps}, we have
\begin{align}
\label{sec6:10}
\j_{\eps}\qty(\mu,\xi)=\j_{0}\qty(\mu,\xi)-\frac{N-2}{2N}\eps\int_{\R^N}k \qty(U_{\mu,\xi}+\phi\qty[\mu,\xi])^{\frac{2N}{N-2}},
\end{align}
where
\begin{align*}
\j_{0}\qty(\mu,\xi)=J_{0}\qty(U_{\mu,\xi}+\phi\qty[\mu,\xi]).
\end{align*}

On the one hand, noticing that
$
	-{\Laplace U_{\mu,\xi}}
	-{\alpha}
	\Riesz{U_{\mu,\xi}^p}
	U_{\mu,\xi}^{p-1}
	=0
$, and using \eqref{sec6:07}, we get
\begin{align*}
J_{0}\qty(U_{\mu,\xi}+\phi\qty[\mu,\xi])
	=J_{0}\qty(U_{\mu,\xi})+\O\qty(\norm{\phi\qty[\mu,\xi]}_{X_{\mu,\xi}}^2),	
\end{align*}
which, together with \eqref{phi:bd} and the fact that $J_{0}\qty(U_{\mu,\xi})=J_{0}\qty(U)$, implies that
\begin{align}
\label{sec6:08}
\j_{0}\qty(\mu,\xi)
	=\bar{\j}_0+\norm{k}_{\infty}^2\O\qty(\eps^2).		
\end{align} 

On the other hand, by \eqref{sec6:07}, we have
\begin{align*}
\int_{\R^N}k \qty(U_{\mu,\xi}+\phi\qty[\mu,\xi])^{\frac{2N}{N-2}}
=
\int_{\R^N}k U_{\mu,\xi}^{\frac{2N}{N-2}}
+
\norm{k}_{\infty}\O\qty(\norm{\phi\qty[\mu,\xi]}_{X_{\mu,\xi}}),	
\end{align*}
which, together with \eqref{phi:bd}, implies that
\begin{align}
\label{sec6:09}
\int_{\R^N}k \qty(U_{\mu,\xi}+\phi\qty[\mu,\xi])^{\frac{2N}{N-2}}
=
\Upsilon\qty(\mu,\xi)
+
\norm{k}_{\infty}^2\O\qty(\eps).		
\end{align}

By inserting \eqref{sec6:08} and \eqref{sec6:09} into \eqref{sec6:10}, we show \eqref{expand:j} holds.
\subsection*{Proof of \eqref{expand:pjmu}} By \eqref{jeps}, we have
\begin{align*}
\pdv{\j_\eps}{\mu}\qty(\mu,\xi)
=J_{\eps}'\qty(\omega\qty[\mu,\xi])\qty(\frac{1}{\mu}Z_{0;\mu,\xi}+\pdv{\phi}{\mu}\qty(\mu,\xi)),
\end{align*}
which, together with \eqref{omega:eq}, implies that
\begin{align}
\label{sec6:15}
\pdv{\j_\eps}{\mu}\qty(\mu,\xi)
=-\sum_{j=0}^N c_{\eps,j}\qty(\mu,\xi)
\int_{\R^N} H_{j;\xi}^{\mu}
\qty(\frac{1}{\mu}Z_{0;\mu,\xi}+\pdv{\phi}{\mu}\qty(\mu,\xi)).	
\end{align}
Therefore, we have
\begin{align*}
\pdv{\j_\eps}{\mu}\qty(\mu,\xi)
=-\frac{1}{\mu}c_{\eps,0}\qty(\mu,\xi)
\int_{\R^N} H_{0}Z_{0}
-\sum_{j=0}^N c_{\eps,j}\qty(\mu,\xi)
\int_{\R^N} H_{j;\xi}^{\mu}\pdv{\phi}{\mu}\qty(\mu,\xi).	
\end{align*}

\subsubsection*{Estimates on $-\frac{1}{\mu}c_{\eps,0}\qty(\mu,\xi)
	\int_{\R^N} H_{0}Z_{0}$}
By \eqref{sec6:N}, we have
\begin{align}
\label{sec6:11}
\int_{\R^N}N\qty(\phi\qty[\mu,\xi];\mu,\xi)Z_{0;\mu,\xi}
=\O\qty(\norm{\phi\qty[\mu,\xi]}_{X_{\mu,\xi}}^2).	
\end{align}
By \eqref{sec6:07}, we get
\begin{align}
\label{sec6:12}
\int_{\R^N} k\qty(U_{\mu,\xi}+\phi\qty[\mu,\xi])^{\frac{N+2}{N-2}}Z_{0;\mu,\xi}
=
\mu \pdv{\Upsilon}{\mu}\qty(\mu,\xi)+
+
\norm{k}_{\infty}\O\qty(\norm{\phi\qty[\mu,\xi]}_{X_{\mu,\xi}}),
\end{align}
where we used the fact that
\begin{align*}
\int_{\R^N}k\qty(x)U_{\mu,\xi}^{\frac{N+2}{N-2}}\qty(x)Z_{0;\mu,\xi}\qty(x)\odif{x}=
\int_{\R^N}k\qty(\mu y+\xi) U^{\frac{N+2}{N-2}}\qty(y)Z_0\qty(y)\odif{y}.
\end{align*}
By inserting \eqref{sec6:11} and \eqref{sec6:12} into \eqref{cj:6},
and using \eqref{phi:bd}, we obtain
\begin{align}
\label{sec6:13}
-\frac{1}{\mu}c_{\eps,0}\qty(\mu,\xi)\int_{\R^N} H_{0}Z_{0}
=
-\eps \pdv{\Upsilon}{\mu}\qty(\mu,\xi)+
\frac{1}{\mu}\norm{k}_{\infty}^2\O\qty(\eps^2).	
\end{align}
\subsubsection*{Estimates on $-\sum_{j=0}^N c_{\eps,j}\qty(\mu,\xi)
\int_{\R^N} H_{j;\xi}^{\mu}\pdv{\phi}{\mu}\qty(\mu,\xi)$} On the one hand, by \eqref{cj:6} and \eqref{phi:bd}, we get 
\begin{align*}
c_{\eps,j}\qty(\mu,\xi) = \norm{k}_{\infty}\O\qty(\eps).
\end{align*}
On the other hand, using \eqref{phi:bd} again, we have
\begin{align*}
\int_{\R^N} H_{j;\xi}^{\mu}\pdv{\phi}{\mu}\qty(\mu,\xi)
=\norm{k}_{\infty}\O\qty(\eps).
\end{align*}
Therefore, we obtain
\begin{align}
\label{sec6:14}
-\sum_{j=0}^N c_{\eps,j}\qty(\mu,\xi)
\int_{\R^N} H_{j;\xi}^{\mu}\pdv{\phi}{\mu}\qty(\mu,\xi)
=\norm{k}_{\infty}^2\O\qty(\eps).
\end{align}
By inserting \eqref{sec6:13} and \eqref{sec6:14} into \eqref{sec6:15}, we get \eqref{expand:pjmu} holds. 
\subsection*{Proof of \eqref{expand:pjxi}} The proof of \eqref{expand:pjxi} is almost identical to that of \eqref{expand:pjmu}, we omit here.

This ends the proof of \Cref{lem:Jexpand}.
\end{proof}

\section{Proof of the main theorem}
\label{sec:pf}
In this section, we complete the proof \Cref{thm:main}. 
First of all, let us state the following lemma without proof. The readers can refer \cite[Section 5.2.2]{AM2006} for more details of the proof.
\begin{lemm}
\label{lem:Upsilon}
Let $\Upsilon$ be given by \eqref{Upsilon}. Suppose $k$ satisfies \ref{k:reg}-\ref{k:L1}. Then the set
\begin{align}
\label{Kscr}
	\mathscr{K}=\setc{\qty(\mu,\xi)\in\qty(0,+\infty)\times\R^N}{\grad_{\mu,\xi}\Upsilon\qty(\mu,\xi)=0}
\end{align}
is compact in $\qty(0,+\infty)\times\R^N$. Moreover, for any open bounded subset $\mathscr{N}$ of $\qty(0,+\infty)\times\R^N$ with $\mathscr{K}\subset \mathscr{N}$, we have
\begin{align}
\label{deg:Kscr}
	\deg\qty( \grad_{\mu,\xi}\Upsilon, \mathscr{N},0 )
	\neq 0.
\end{align}
\end{lemm}
Now, we are able to prove our main result.
\begin{proof}[Proof of \Cref{thm:main}]
Let the set $\mathscr{K}$ be given by \eqref{Kscr}. 
By \Cref{lem:Upsilon}, we obtain that the set $\mathscr{K}$ is a compact subset of $\qty(0,+\infty)\times\R^N$. Therefore, we have
\begin{align*}
	\munderbar{\mu}:=\inf_{\qty(\mu,\xi)\in\mathscr{K} }\mu>0.
\end{align*}
Moreover, there exists an open set $\mathscr{N}_0$ satisfying 
$\mathscr{K}\subset \mathscr{N}_0\subset \qty(0,+\infty)\times\R^N$ such that
\begin{align}
\label{sec6:16}
\inf_{\qty(\mu,\xi)\in\bar{\mathscr{N}}_0 }\mu\geq \frac{\munderbar{\mu}}{2}.
\end{align}

By the definition of $\mathscr{K}$, we easily obtain that
\begin{align}
\label{sec6:17}
\munderbar{m}=\inf_{\qty(\mu,\xi)\in\partial\mathscr{N}_0}\abs{\grad_{\mu,\xi}\Upsilon\qty(\mu,\xi)}>0,
\end{align}
which, together with \eqref{sec6:16}, \eqref{expand:pjmu} and \eqref{expand:pjxi}, implies that, for all $\eps>0$ sufficiently small,
\begin{align*}
	\inf_{\qty(\mu,\xi)\in\partial\mathscr{N}_0}\abs{\frac{1}{\eps}\grad_{\mu,\xi}\j_{\eps}\qty(\mu,\xi)}>0.
\end{align*}
Hence, by \eqref{expand:pjmu} and \eqref{expand:pjxi}, and using \eqref{deg:Kscr}, we get, for all $\eps>0$ sufficiently small,
\begin{align*}
\deg\qty( -\frac{1}{\eps}\grad_{\mu,\xi}\j_{\eps}, \mathscr{N}_0,0 )
=
\deg\qty( \grad_{\mu,\xi}\Upsilon, \mathscr{N}_0,0 )
\neq 0,
\end{align*}
which implies that there exists $\qty(\mu_{\eps},\xi_{\eps})\in\mathscr{N}_0 $ such that
\begin{align*}
	\grad_{\mu,\xi}\j_{\eps}\qty(\mu_{\eps},\xi_{\eps})=0.
\end{align*}
By \Cref{lem:c2J}, we obtain that $U_{\mu_{\eps},\xi_{\eps}}+\phi\qty[\mu_{\eps},\xi_{\eps}]$ is a solution of \eqref{NLH-eps}. This ends the proof of \Cref{thm:main}.	
\end{proof}	
	

\noindent \subsection*{Acknowledgements.}
G. Xu  was supported by National Key Research and Development Program of China (No. 2020YFA0712900) and by NSFC (No. 11831004). G. Lv was supported by the NSF of China grants 12171247, Jiangsu Provincial Double-Innovation Doctor Program JSSCBS20210466 and Qing Lan Project, and Postgraduate Research and Practice Innovation Program of Jiangsu Province (No. KYCX21 0932).
X. Tang was supported by NSFC (No. 12001284).
\appendix
\section{Useful facts related to the bubble $U$}
In this section, we collect several facts related the bubble $U$. First of all, let us introduce the following well-known estimates. we refer to, for example, \cite{TX2022,Lieb1983,GMYZ2022} for proofs of these results.
\begin{lemm}
\label{lem:decay}
Let $\lambda\in\qty({0,N})$ and $\theta\in\qty(N,+\infty)$. Then there exists a constant $C>0$ such that
\begin{equation}
	\label{decay all}
	\int_{\R^N}\frac{1}{\abs{x-y}^{\lambda}}\frac{1}{\jp{y}^{\theta}}\odif{y}
	\leq \frac{C}{\jp{x}^{\lambda}},
\end{equation}
especially, 
\begin{align}
	\label{A:id}
	\int_{\R^N}\frac{1}{\abs{x-y}^{\lambda}}\frac{1}{\jp{y}^{2N-\lambda}}\odif{y}
	=\pi^{\frac{N}{2}}\frac{\Gamma\qty(\frac{N-\lambda}{2})}{\Gamma\qty(\frac{2N-\lambda}{2})}\frac{1}{\jp{x}^{\lambda}}.
\end{align}		
\end{lemm}

Recall that
\begin{equation*}
\fct{U}{x}
=\frac{1}{\qty({1+\abs{x}^2})^{\frac{N-2}{2}}},
\end{equation*}
and
\begin{align}
\label{A:U:prmt}
U_{\mu,\xi}\qty(x)
=\frac{1}{\mu^{\frac{N-2}{2}}}U\qty(\frac{x-\xi}{\mu})
,
\text{~~where~~} \qty(\mu,\xi)\in\qty(0,+\infty)\times\R^N.
\end{align}
Let us list several identities related to the bubble $U$.
\begin{lemm}
\label{lem:identities}
Let the functions $Z_j$ ($j=0,1,\ldots,N$) and $H_m$ ($m=0,1,\ldots,N$) be given by \eqref{Zj} and \eqref{Hj} respectively. Then we have:
\begin{gather}
\label{H2Z:int}
{Z_j}\qty(x)=\frac{\qty(N+2)}{\qty(N-2)\omega_N}\int_{\R^N}
\frac{\fct{H_j}{y}}{\abs{x-y}^{N-2}}\odif{y},
\\
\label{ZjHm:vns}
\int_{\R^N} Z_j H_m =0, \qq{with} j\neq m,
\\
\label{Z0H0}
\int_{\R^N} Z_0 H_0 = \frac{\qty(N-2)^2}{4}\int_{\R^N}\frac{\qty(\abs{x}^2-1)^2}{\jp{x}^{2N+4}}\odif{x},
\\
\label{ZjHj}
\int_{\R^N} Z_j H_j =\frac{\qty(N-2)^2}{N}\int_{\R^N}\frac{\abs{x}^2}{\jp{x}^{2N+4}}\odif{x},
\qq{with} j=1,2,\ldots,N.	
\end{gather}	
\end{lemm}
\begin{proof}
Recall, by direct calculations, for any $\qty(\mu,\xi)\in\qty(0,+\infty)\times\R^N$, 
\begin{align}
	\label{U:id:prmt}
	-\Laplace U_{\mu,\xi}\qty(x)=N\qty(N-2)U^{\frac{N+2}{N-2}}_{\mu,\xi}\qty(x).
\end{align}
By differentiating \eqref{U:id:prmt} with respect to $\mu$ and $\xi_j$~($j=1,2,\ldots,N$) at $\mu=1$ and $\xi=0$, we get,
for $j=0,1,\ldots,N$,
\begin{equation}
	\label{H2Z}
	-\Laplace Z_j = N\qty(N+2)H_j,
\end{equation}
which, together with \Cref{lem:possion}, implies that \eqref{H2Z:int} holds.

Next, by \eqref{Zj} and \eqref{Hj}, we easily check that \eqref{ZjHm:vns}, \eqref{Z0H0} hold, and 
\begin{align*}
	\int_{\R^N} Z_j H_j =\int_{\R^N} Z_1 H_1, 
	\qq{with} j=2,3,\ldots,N.
\end{align*}
Hence, we have
\begin{align*}
	\int_{\R^N}H_j Z_j 
	= \frac{1}{N}\sum_{i=1}^N\int_{\R^N}H_i Z_i
	=\frac{\qty(N-2)^2}{N}\int_{\R^N}\frac{\abs{x}^2}{\jp{x}^{2N+4}}\odif{x},
\end{align*}
which implies that \eqref{ZjHj} holds. This ends the proof of \Cref{lem:identities}.	
\end{proof}
By direct calculations, we verify that
\begin{lemm}
Let $\bar{Z}_{m,j}$ ( $m=0,1,\ldots,N$, $j=0,1,\ldots,N$) be given by \eqref{Zbar}. Then, we have
\begin{gather*}
\bar{Z}_{0,{0}}\qty(x)
=
\frac{N\qty(N-2)}{4}
\frac{ \qty(\abs{x}^2-1 )^2 }{ \jp{x}^{N+2} }
-
\frac{N-2}{2}\frac{1}{\jp{x}^{N-2}},
\\
\bar{Z}_{0,{j}}\qty(x)
=
\frac{ N\qty(N-2)}{2}\frac{\qty(\abs{x}^2-1)x_j}{\jp{x}^{N+2}}
-
\qty(N-2)\frac{x_j}{\jp{x}^N},
\\
\bar{Z}_{m,j}\qty(x)
=
-\qty(N-2)\frac{\delta_{j,m}}{\jp{x}^N}
-\frac{N\qty(N-2)}{2}\frac{x_m x_j}{\jp{x}^{N+2}},	
\end{gather*}
where $j=1,\ldots,N$, $m=1,\ldots,N$, and
$
\delta_{j,m}=
\begin{cases}
	1,& j=m,
	\\
	0, & j\neq m.
\end{cases}	
$
\end{lemm}

Next, let us turn our attention to the function $U_{\mu,\xi}$.
By a simple change of variables, using \eqref{A:id}, 
we easily obtain that
\begin{align}
\label{A:id:prmt}
\Riesz{U_{\mu,\xi}^{p}}\qty(x)
=
\pi^{\frac{N}{2}}\frac{\Gamma\qty(\frac{N-\lambda}{2})}{\Gamma\qty(\frac{2N-\lambda}{2})}
U_{\mu,\xi}^{\frac{\lambda}{N-2}}\qty(x).	
\end{align}
Moreover, we have  
\begin{lemm}
For any $\qty(\mu,\xi)\in\qty(0,+\infty)\times\R^N$, we have
\begin{align}
\label{tHbZ:id}
\int_{\R^N}\widetilde{H}^{\mu}_{m,j;\xi}Z_{j;\mu,\xi}
=
-
\int_{\R^N}H^{\mu}_{m;\xi}\bar{Z}_{m,j;\mu,\xi},
\qq{where} m,j=0,1,\ldots,N.		
\end{align}
where $H_m$, $Z_m$, $\bar{Z}_{m,j}$ and $\widetilde{H}_{m,j}$ be given by \eqref{Hj}, \eqref{Zj}, \eqref{Zbar} and \eqref{Htilde} respectively.
\end{lemm} 
\begin{proof}
By a change of variables, we have
\begin{align*}
\int_{\R^N}H^{\mu}_{m;\xi}Z_{m;\mu,\xi}
=
\int_{\R^N}H_{m}Z_{m},
\text{~for any~} \mu\in\qty(0,+\infty),~~\xi\in\R^N.	
\end{align*}
By taking partial derivative with respect to $\mu$ and $\xi_j~\qty(j=1,2\ldots,N)$ in the above equation, we get that
\eqref{tHbZ:id} holds.
\end{proof}
By a standard Taylor expansion argument, we have the following lemma. 
\begin{lemm}
Let $\theta,~\beta\in\R$. Then
\begin{gather}
\label{A:a-ineq}
\qty(1+a)^{\theta}=1+\theta a+\O\qty(a^2),
\qq{as} a\to 0,
\\
\label{A:ab-ineq}
\qty(1+a)^{\theta}\qty(1+b)^{\beta}=1+\theta a + \beta b +
\O\qty( a^2+b^2 ),
\qq{as} a\to 0 \text{~and~} b\to 0.	
\end{gather}	
\end{lemm}

We end this section with the following estimates.
\begin{lemm}
\label{lem:U:expand}
Let $\qty(\mu_0,\xi_0)\in\qty(0,+\infty)\times\R^N$ be fixed. 
Then, as $\qty(\mu,\xi)\to\qty(\mu_0,\xi_0)$, we have 
\begin{align}
\label{U:expand}
\notag
U_{\mu,\xi} \qty(x)
=&
U_{\mu_0,\xi_0} \qty(x)
+ Z_{0;\mu_0,\xi_0}\qty(x)\frac{\mu-\mu_0}{\mu_0}
+
\sum_{j=1}^{N} Z_{j;\mu_0,\xi_0}\qty(x)\frac{\xi_j-\xi_{0j}}{\mu_0}
\\
&+
U_{\mu_0,\xi_0}\qty(x)
\O\qty( \frac{\abs{\mu-\mu_0}^2}{\mu_0^2}
+\frac{1}{\jp{\frac{x-\xi_0}{\mu_0}}^2}\frac{\abs{\xi-\xi_0}^2}{\mu_0^2} ),
\end{align}
\begin{align}
\notag
\label{Zj:expand}
Z_{j;\mu,\xi} \qty(x)
=&
Z_{j;\mu_0,\xi_0} \qty(x)
+ \bar{Z}_{j,0;\mu_0,\xi_0}\qty(x)\frac{\mu-\mu_0}{\mu_0}
+
\sum_{m=1}^{N} Z_{j,m;\mu_0,\xi_0}\qty(x)\frac{\xi_m-\xi_{0m}}{\mu_0}
\\
&+
U_{\mu_0,\xi_0}\qty(x)
\O\qty( \frac{\abs{\mu-\mu_0}^2}{\mu_0^2}
+\frac{1}{\jp{\frac{x-\xi_0}{\mu_0}}^2}\frac{\abs{\xi-\xi_0}^2}{\mu_0^2} ),
\end{align}
\begin{align}
\notag
\label{Hm:expand}
H^{\mu}_{m;\xi} \qty(x)
=&
H^{\mu_0}_{m;\xi_0} \qty(x)
+ \widetilde{H}^{\mu_0}_{m,0;\xi_0}\qty(x)\frac{\mu-\mu_0}{\mu_0}
+
\sum_{m=1}^{N} \widetilde{H}^{\mu_0}_{m,j;\mu_0,\xi_0}\qty(x)\frac{\xi_j-\xi_{0}}{\mu_0}
\\
&+
U_{\mu_0,\xi_0}^{\frac{N+2}{N-2}}\qty(x)
\O\qty( \frac{\abs{\mu-\mu_0}^2}{\mu_0^2}
+\frac{1}{\jp{\frac{x-\xi_0}{\mu_0}}^2}\frac{\abs{\xi-\xi_0}^2}{\mu_0^2} ).	
\end{align}

\end{lemm}
\begin{proof}
We only verify \eqref{U:expand} holds. The expression \eqref{Zj:expand} and \eqref{Hm:expand} can be obtained in a similar manner. 

First, we claim that
\begin{align}
\label{U0:expand}
U_{\nu,\zeta}\qty(y)
=U\qty(y)
+Z_0\qty(y)\qty(\nu-1)
+\sum_{m=1}^N Z_m\qty(y)\zeta_m
+\frac{1}{\jp{y}^{N-2}}
\O\qty( \abs{\nu-1}^2+\frac{\abs{\zeta}^2}{\jp{y}^2} ),	
\end{align}
as $\nu\to 1$ and $\zeta\to  0$.

Indeed, using the facts
\begin{align*}
	\frac{1}{\nu}=\frac{1}{1+\qty(\nu-1)}
	=
	1-\qty(\nu-1)+\O\qty(\qty(\nu-1)^2),
\qq{as} \nu \to 1. 	
\end{align*}
and
\begin{align*}
	\frac{\abs{y-\zeta}^2}{\mu}
	=
	\abs{y}^2-2y\zeta-\qty(\nu-1)\abs{y}^2
	+ \O\qty({\jp{y}^2}\qty(\nu-1)^2+{\abs{\zeta}^2}),
~~\text{as}~~ \nu \to 1, \zeta\to  0,
\end{align*}
we get
\begin{align*}
\nu+\frac{\abs{y-\zeta}^2}{\nu}
=
\jp{y}^2
\qty( 1 -\qty(\nu-1)\frac{\abs{y}^2-1}{\jp{y}^2}
-2\zeta\frac{y}{\jp{y}^2}
+\O\qty(\qty(\nu-1)^2+\frac{\abs{\zeta}^2}{\jp{y}^2}) ),		
\end{align*}
as $\nu\to 1$ and $\zeta\to  0$. Therefore, by \eqref{A:a-ineq}, we have
\begin{align*}
\qty(\frac{1}{\nu+\frac{\abs{y-\zeta}^2}{\nu}})^{\frac{N-2}{2}}
=
&
\frac{1}{\jp{y}^{N-2}}
+
\frac{N-2}{2}\frac{1}{\jp{y}^{N-2}}\frac{\abs{y}^2-1}{\jp{y}^2}\qty(\nu-1)
\\
+&
\qty(N-2)\frac{1}{\jp{y}^{N-2}}\frac{y\zeta}{\jp{y}^2}
+
\frac{1}{\jp{y}^{N-2}}
\o\qty(\qty(\nu-1)^2+\frac{\abs{\zeta}^2}{\jp{y}^2}),	
\end{align*}
which implies that \eqref{U0:expand} holds.

Next, by letting $\nu=\frac{\mu}{\mu_0}$ and $\zeta = \frac{\xi-\xi_0}{\mu_0}$ and $y=\frac{x-\xi_0}{\mu_0}$ in \eqref{U0:expand}, and multiplying \eqref{U0:expand} with $\frac{1}{\mu_0^{\frac{N-2}{2}}}$, we get \eqref{U:expand} holds. This ends the proof of \Cref{lem:U:expand}.
\end{proof}	

\section{Proof of \Cref{lemm:C1}}
\label{app:C1}
Let $\qty(\mu_0,\xi_0)\in\qty(0,+\infty)\times\R^N$ be fixed. Throughout \Cref{app:C1}, we always assume $(\mu,\xi)\in \qty(0,+\infty)\times \R^N$ is sufficiently close to $\qty(\mu_0,\xi_0)$ such that
\begin{align*}
\frac{1}{2}\mu_0\jp{\frac{x-\xi_0}{\mu_0}}^2
<\mu\jp{\frac{x-\xi}{\mu}}
<2\mu_0\jp{\frac{x-\xi_0}{\mu_0}}^2,
~~x\in\R^N.	
\end{align*}
By direct computations, we easily check that
\begin{gather}
\label{B:phi:bd}
\frac{1}{2^{N-2}}\norm{\varphi}_{X_{\mu_0,\xi_0}}
<\norm{\varphi}_{X_{\mu ,\xi }}<{2^{N-2}}\norm{\varphi}_{X_{\mu_0,\xi_0}},	
\qq{for any} \varphi\in X_{\mu_0,\xi_0},
\\
\label{B:g:bd}
\frac{1}{2^{N+2}}\norm{g}_{Y_{\mu_0,\xi_0}}
<\norm{g}_{Y_{\mu ,\xi }}<{2^{N+2}}\norm{g}_{Y_{\mu_0,\xi_0}},
\qq{for any} g\in Y_{\mu_0,\xi_0}.	
\end{gather}
For later use, for any $\varphi\in X_{\mu,\xi}$, 
let us introduce the notation
\begin{gather}
\label{B:OP_proj}
\Pi_{\mu,\xi}^{\perp}\qty(\varphi)
=
\sum_{j=0}^N\dfrac{\displaystyle\int_{\R^N}\varphi H_{j;\xi}^{\mu}  }{\displaystyle\int_{\R^N}H_j Z_{j}} Z_{j;\mu,\xi},
\qq{and}
\Pi_{\mu,\xi}\qty(\varphi)=\varphi-\Pi_{\mu,\xi}^{\perp}\qty(\varphi).	
\end{gather}

Let $g\in Y_{\mu_0,\xi_0}$ be fixed.
Noticing \eqref{B:g:bd}, we have $g\in Y_{\mu,\xi}$. 
Therefore, by \Cref{prop:linsltn:prmt}, there exist 
$\phi\qty[\mu_0,\xi_0]=\T\qty[{\mu_0,\xi_0}]g\in X_{\mu_0,\xi_0}$ 
and
$\phi\qty[\mu,\xi]=\T\qty[{\mu,\xi}]g\in X_{\mu,\xi}$ 
solving the problems
\begin{align}
\label{b:prob0}
\begin{dcases}
	{\mathscr{L}_{\mu_0,{\xi_0}}\phi\qty[\mu_0,\xi_0]} = g-
	\sum_{j=0}^{N}\dfrac{\displaystyle\int_{\R^N}g Z_{j;\mu_0,\xi_0}}{\displaystyle\int_{\R^N}H_jZ_j }
	H_{j;\mu_0}^{\xi_0},
	\\
	\int_{\R^N}{\phi\qty[\mu_0,\xi_0]}{H_{j;\xi_0}^{\mu_0}}=0,\qq{where}j=0,1,\ldots,N,	
\end{dcases}								
\end{align}
and
\begin{align}
\label{b:prob}
\begin{dcases}
	{\mathscr{L}_{\mu,{\xi}}\phi\qty[\mu,\xi]} = g-
	\sum_{j=0}^{N}\dfrac{\displaystyle\int_{\R^N}g Z_{j;\mu,\xi}}{\displaystyle\int_{\R^N}H_jZ_j}H_{j;\xi}^{\mu},
	\\
	\int_{\R^N}{\phi\qty[\mu,\xi]}H_{j;\xi}^{\mu}=0,\qq{where}j=0,1,\ldots,N,
\end{dcases}							
\end{align}
respectively. Moreover, the following estimates hold
\begin{gather}
\label{B:g:phi:bd0}
	\norm{\phi\qty[\mu_0,\xi_0]}_{X_{\mu_0,\xi_0}}
	\leq C\norm{g}_{Y_{\mu_0,\xi_0}},
\\
\label{B:02}
\norm{\phi\qty[\mu,\xi]}_{X_{\mu,\xi}}
\leq C\norm{g}_{Y_{\mu,\xi}}.
\end{gather}
By inserting \eqref{B:phi:bd} and \eqref{B:g:bd} into \eqref{B:02}, we get 
\begin{align}
\label{B:g:phi:bd}
	\norm{\phi\qty[\mu,\xi]}_{X_{\mu_0,\xi_0}}
	\leq C\norm{g}_{Y_{\mu_0,\xi_0}}.
\end{align}
\subsection*{Continuity of $\T\qty[{\mu,\xi}]$} We first show that, for any $g\in Y_{\mu,\xi}$, the function $\phi\qty[\mu,\xi]=\T\qty[{\mu,\xi}]g$ is continuous with respect to $\mu$ and $\xi$. More precisely, we prove
\begin{lemm}
\label{B:lemm:c}
For any $g\in Y_{\mu_0,\xi_0}$, 
let $\phi\qty[\mu_0,\xi_0]=\T\qty[{\mu_0,\xi_0}]g$ and $\phi\qty[\mu,\xi]=\T\qty[{\mu,\xi}]g$
be given by \Cref{prop:linsltn:prmt}, i.e. the functions
$\phi\qty[\mu_0,\xi_0]$ and $\phi\qty[\mu,\xi]$ are solutions of the problems \eqref{b:prob0} and \eqref{b:prob} respectively.
Then 
\begin{align}
\label{b:cvg}
	\lim_{
		\substack{\mu\to \mu_0\\ \xi\to\xi_0}}
	\norm{\phi\qty[\mu,\xi]-\phi\qty[\mu_0,\xi_0] }_{X_{\mu_0,\xi_0}}=0.		
\end{align}	
\end{lemm}
\begin{proof}
First, by \eqref{B:OP_proj}, the function $\phi\qty[\mu,\xi]-\phi\qty[\mu_0,\xi_0]$ can be decomposed as follows
\begin{align}
\label{B:14}
	\phi\qty[\mu,\xi]-\phi\qty[\mu_0,\xi_0]=
\Pi_{\mu_0,\xi_0}\qty(\phi\qty[\mu,\xi]-\phi\qty[\mu_0,\xi_0])
+
\Pi_{\mu_0,\xi_0}^{\perp}\qty(\phi\qty[\mu,\xi]-\phi\qty[\mu_0,\xi_0]).
\end{align}
Hence, in order to show that \eqref{b:cvg} holds, we only need to show that
\begin{gather}
\label{b:cvg:PIperp}
\lim_{
	\substack{\mu\to \mu_0\\ \xi\to\xi_0}}
\norm{\Pi_{\mu_0,\xi_0}^{\perp}\qty(\phi\qty[\mu,\xi]-\phi\qty[\mu_0,\xi_0]) }_{X_{\mu_0,\xi_0}}=0,\\
\label{b:cvg:PI}
\lim_{
	\substack{\mu\to \mu_0\\ \xi\to\xi_0}}
\norm{\Pi_{\mu_0,\xi_0}\qty(\phi\qty[\mu,\xi]-\phi\qty[\mu_0,\xi_0]) }_{X_{\mu_0,\xi_0}}=0.
\end{gather}
\subsection*{Proof of \eqref{b:cvg:PIperp}}
Since 
\begin{align*}
\Pi_{\mu_0,\xi_0}^{\perp}\qty(\phi\qty[\mu,\xi]-\phi\qty[\mu_0,\xi_0])
=
\sum_{j=0}^N\dfrac{\displaystyle\int_{\R^N}
\qty(\phi\qty[\mu,\xi]-\phi\qty[\mu_0,\xi_0])	
 H_{j;\xi_0}^{\mu_0}  }{\displaystyle\int_{\R^N}H_j Z_{j}} Z_{j;\mu_0,\xi_0},
\end{align*}
from $\int_{\R^N}{\phi\qty[\mu_0,\xi_0]}{H_{j;\xi_0}^{\mu_0}}=0$
and
$\int_{\R^N}{\phi\qty[\mu,\xi]}H_{j;\xi}^{\mu}=0$~ ($j=0,1,\ldots,N$),
we deduce that
\begin{align}
\label{B:03}
\Pi_{\mu_0,\xi_0}^{\perp}\qty(\phi\qty[\mu,\xi]-\phi\qty[\mu_0,\xi_0])
=
-\sum_{j=0}^N\dfrac{\displaystyle\int_{\R^N}
\phi\qty[\mu,\xi]\qty(H_{j;\xi}^{\mu}-  H_{j;\xi_0}^{\mu_0})}{
\displaystyle\int_{\R^N}H_j Z_{j}} Z_{j;\mu_0,\xi_0}.	
\end{align}
Using \eqref{Hm:expand}, we obtain that, as $\qty(\mu,\xi)\to\qty(\mu_0,\xi_0)$,
\begin{align*}
\norm{H_{j;\xi}^{\mu}-  H_{j;\xi_0}^{\mu_0}}_{Y_{\mu_0,\xi_0}}
=\O\qty(
\frac{\abs{\mu-\mu_0}}{\mu_0}+\frac{\abs{\xi-\xi_0}}{\mu_0}
),
\end{align*}
which, together with \eqref{B:g:phi:bd}, implies that
\begin{align*}
\abs{
\int_{\R^N}
\phi\qty[\mu,\xi]\qty(H_{j;\xi}^{\mu}-  H_{j;\xi_0}^{\mu_0})
}
=\norm{g}_{Y_{\mu_0,\xi_0}}	
\O\qty(
\frac{\abs{\mu-\mu_0}}{\mu_0}+\frac{\abs{\xi-\xi_0}}{\mu_0}
).
\end{align*}
as $\qty(\mu,\xi)\to\qty(\mu_0,\xi_0)$.
Therefore, by \eqref{B:03}, we have, as $\qty(\mu,\xi)\to\qty(\mu_0,\xi_0)$,
\begin{align}
\label{B:13}
\norm{\Pi_{\mu_0,\xi_0}^{\perp}\qty(\phi\qty[\mu,\xi]-\phi\qty[\mu_0,\xi_0])}_{X_{\mu_0,\xi_0}}
=
\norm{g}_{Y_{\mu_0,\xi_0}}	
\O\qty(
\frac{\abs{\mu-\mu_0}}{\mu_0}+\frac{\abs{\xi-\xi_0}}{\mu_0}
),
\end{align}
which implies that \eqref{b:cvg:PIperp} holds.

\subsection*{Proof of \eqref{b:cvg:PI}}
Using the fact that $\mathscr{L}_{\mu_0,{\xi_0}}Z_{j;\mu_0,\xi_0}=0$~ ($j=0,1,\ldots,N$), we deduce from \eqref{b:prob0} and \eqref{b:prob} that 
\begin{align}
\label{B:07}
\begin{dcases}
\mathscr{L}_{\mu_0,{\xi_0}}\Pi_{\mu_0,\xi_0}\qty(\phi\qty[\mu,\xi]-\phi\qty[\mu_0,\xi_0])
=
\delta_{\mu,\xi}l\qty(\phi\qty[\mu,\xi])
-
\delta_{\mu,\xi}r\qty(g)
+
\delta_{\mu,\xi}t\qty(g), 
\\
\int_{\R^N}	\Pi_{\mu_0,\xi_0}\qty(\phi\qty[\mu,\xi]-\phi\qty[\mu_0,\xi_0])H_{j;\xi_0}^{\mu_0}=0,						
\end{dcases}							
\end{align}
where
\begin{gather*}
\begin{align*}
	\delta_{\mu,\xi}l\qty(\phi\qty[\mu,\xi])
	=
	\alpha p \qty(\Riesz{U_{\mu,\xi}^{p-1}\phi\qty[\mu,\xi]}U_{\mu,\xi}^{p-1}
	- \Riesz{U_{\mu_0,\xi_0}^{p-1}\phi\qty[\mu,\xi]}U_{\mu_0,\xi_0}^{p-1})
	\\
	+
	\alpha (p-1)
	\qty( \Riesz{U_{\mu,\xi}^{p}}U_{\mu,\xi}^{p-2}\phi\qty[\mu,\xi]
	- \Riesz{U_{\mu_0,\xi_0}^{p}}U_{\mu_0,\xi_0}^{p-2}\phi\qty[\mu,\xi]),				
\end{align*}
\\
\delta_{\mu,\xi}r\qty(g)
=
\sum_{j=0}^{N}
\dfrac{\displaystyle\int_{\R^N}g Z_{\mu,\xi,j}}{\displaystyle\int_{\R^N}{H_jZ_{j}}}
\qty(H_{j;\xi}^{\mu}-H_{j;\xi_0}^{\mu_0}),
\\
\delta_{\mu,\xi}t\qty(g)
=
-\sum_{j=0}^{N}
\dfrac{{\displaystyle\int_{\R^N}g \qty(Z_{j;\mu,\xi}-Z_{j;\mu_0,\xi_0})}}{\displaystyle\int_{\R^N}{H_jZ_{j}}}H_{j;\xi_0}^{\mu_0}.	
\end{gather*}

Using \eqref{U:expand} and \eqref{A:ab-ineq}, we deduce that,
as $\qty(\mu,\xi)\to\qty(\mu_0,\xi_0)$,
\begin{align*}
	\norm{l\qty(\phi\qty[\mu,\xi])}_{Y_{\mu_0,\xi_0}}
	=
	\norm{\phi\qty[\mu,\xi]}_{X_{\mu_0,\xi_0}}
	\O\qty( \frac{\abs{\mu-\mu_0}}{\mu_0}
	+
	\frac{\abs{\xi-\xi_0}}{\mu_0} ),
\end{align*}
which, together with \eqref{B:g:phi:bd}, implies that,
\begin{align}
\label{B:04}
\norm{l\qty(\phi\qty[\mu,\xi])}_{Y_{\mu_0,\xi_0}}
=
\norm{g}_{Y_{\mu_0,\xi_0}}
\O\qty( \frac{\abs{\mu-\mu_0}}{\mu_0}
+
\frac{\abs{\xi-\xi_0}}{\mu_0} ).	
\end{align}
as $\qty(\mu,\xi)\to\qty(\mu_0,\xi_0)$.

By \eqref{Hm:expand}, we have, as $\qty(\mu,\xi)\to\qty(\mu_0,\xi_0)$,
\begin{align*}
\norm{H_{j;\xi}^{\mu}-H_{j;\xi_0}^{\mu_0}}_{Y_{\mu_0,\xi_0}}=\O\qty( \frac{\abs{\mu-\mu_0}}{\mu_0}
+
\frac{\abs{\xi-\xi_0}}{\mu_0} ),
\end{align*}
which, together with the fact that $\norm{Z_{j;\mu,\xi}}_{X_{\mu_0,\xi_0}}\leq C$, implies that,
\begin{align}
\label{B:05}
\norm{\delta_{\mu,\xi}r\qty(g)}_{Y_{\mu_0,\xi_0}}
=
\norm{g}_{Y_{\mu_0,\xi_0}}
\O\qty( \frac{\abs{\mu-\mu_0}}{\mu_0}
+
\frac{\abs{\xi-\xi_0}}{\mu_0} ).
\end{align}
as $\qty(\mu,\xi)\to\qty(\mu_0,\xi_0)$.

By \eqref{Zj:expand}, we get, as $\qty(\mu,\xi)\to\qty(\mu_0,\xi_0)$,
\begin{align*}
\norm{Z_{j;\mu,\xi}-Z_{j;\mu_0,\xi_0}}_{X_{\mu_0,\xi_0}}=\O\qty( \frac{\abs{\mu-\mu_0}}{\mu_0}
+
\frac{\abs{\xi-\xi_0}}{\mu_0} ),	
\end{align*}
which, together with the fact that $\norm{H_{j;\xi}^{\mu}}_{Y_{\mu_0,\xi_0}}\leq C$, implies that
\begin{align}
\label{B:06}
\norm{\delta_{\mu,\xi}t\qty(g)}_{Y_{\mu_0,\xi_0}}
=
\norm{g}_{Y_{\mu_0,\xi_0}}
\O\qty( \frac{\abs{\mu-\mu_0}}{\mu_0}
+
\frac{\abs{\xi-\xi_0}}{\mu_0} ),	
\end{align}
as $\qty(\mu,\xi)\to\qty(\mu_0,\xi_0)$.

Using \Cref{prop:linsltn:prmt}, by inserting \eqref{B:04}, \eqref{B:05} and \eqref{B:06} into \eqref{B:07}, we get,
as $\qty(\mu,\xi)\to\qty(\mu_0,\xi_0)$,
\begin{align}
\label{B:12}
\norm{\Pi_{\mu_0,\xi_0}\qty(\phi\qty[\mu,\xi]-\phi\qty[\mu_0,\xi_0])}_{X_{\mu_0,\xi_0}}=
\norm{g}_{Y_{\mu_0,\xi_0}}
\O\qty( \frac{\abs{\mu-\mu_0}}{\mu_0}
+
\frac{\abs{\xi-\xi_0}}{\mu_0} ),
\end{align} 
hence \eqref{b:cvg:PI} holds.

This ends the proof of \Cref{B:lemm:c}.	
\end{proof}

Next, in order to concern with the differentiability of $\phi_{\mu,\xi}$ with respect to $\mu$ and $\xi$, we introduce the following problem
\begin{align}
	\label{B:psi:eq}
	\begin{dcases}
		{\mathscr{L}_{\mu_0,{\xi_0}}\widetilde{\psi}_m}
		=
		\frac{1}{\mu_0}
		l_m\qty(\phi\qty[\mu_0,\xi_0])
		-\frac{1}{\mu_0}r_m\qty(g)
		+\frac{1}{\mu_0}t_m\qty(g),
		\\
		\int_{\R^N}\widetilde{\psi}_m H_{j;\xi_0}^{\mu_0}=0,		
	\end{dcases}		
\end{align}
where $m=0,1,\ldots,N$,
\begin{align}
\label{B:lm}
\notag
l_m\qty(\phi\qty[\mu_0,\xi_0])
=&
\alpha p\qty(p-1)
\Riesz{U_{\mu_0,\xi_0}^{p-1}\phi\qty[\mu_0,\xi_0]}U_{\mu_0,\xi_0}^{p-2}Z_{m;\mu_0,\xi_0}
\\
\notag
&
+
\alpha p\qty(p-1)
\Riesz{U_{\mu_0,\xi_0}^{p-2}Z_{m;\mu_0,\xi_0}\phi\qty[\mu_0,\xi_0]}U_{\mu_0,\xi_0}^{p-1}	
\\
\notag
&+
\alpha p\qty(p-1)
\Riesz{U_{\mu_0,\xi_0}^{p-1}Z_{m;\mu_0,\xi_0}}U_{\mu_0,\xi_0}^{p-2}\phi\qty[\mu_0,\xi_0]
\\
&
+
\alpha \qty(p-1)\qty(p-2)
\Riesz{U_{\mu_0,\xi_0}^{p}}U_{\mu_0,\xi_0}^{p-3}Z_{m;\mu_0,\xi_0}\phi\qty[\mu_0,\xi_0],			
\end{align}
\begin{gather}
\label{B:rm}
r_m\qty(g)
=
\sum_{j=0}^N\dfrac{\displaystyle\int_{\R^N} g Z_{j;\mu_0,\xi_0}}{\displaystyle\int_{\R^N}H_jZ_j}\widetilde{H}_{j,m;\mu_0,\xi_0},	
\end{gather}
and
\begin{gather}
\label{B:tm}
t_m\qty(g)
=
-\sum_{j=0}^N\dfrac{\displaystyle\int_{\R^N} g \bar{Z}_{j,m;\mu_0,\xi_0}}{\displaystyle\int_{\R^N}H_jZ_j}{H}_{j;\mu_0,\xi_0}.		
\end{gather}
\begin{lemm}
\label{B:lemm:psi:td}
Let $g\in Y_{\mu_0,\xi_0}$ and $\phi\qty[\mu_0,\xi_0]$ be the solution to \eqref{b:prob0}.
For each $m=0,1,\ldots,N$, there exists a unique $\widetilde{\psi}_m$ solves \eqref{B:psi:eq}. Moreover, we have
\begin{align*}
	\norm{\widetilde{\psi}_m}_{X_{\mu_0,\xi_0}}
	\leq \frac{C}{\mu_0}\norm{g}_{Y_{\mu_0,\xi_0}},
\quad
m=0,1,\ldots,N.
\end{align*}
\end{lemm}
\begin{proof}
We only show the case $m=0$.
First, since
\begin{align*}
-{\Laplace Z_{j;\mu,\xi} }
-{\alpha}p\Riesz{U_{\mu,\xi}^{p-1}Z_{j;\mu,\xi}}U_{\mu,\xi}^{p-1}
-
{\alpha}\qty(p-1)\Riesz{U_{\mu,\xi}^p}U_{\mu,\xi}^{p-2}Z_{j;\mu,\xi}
=0,
\end{align*}
by differentiating the above equation with respect to $\mu$
as $\qty(\mu_0,\xi_0)$, we get
\begin{align}
\label{B:08}
\Lscr_{\mu_0,\xi_0}\bar{Z}_{j,0;\mu_0,\xi_0}
=
l_0\qty(Z_{j;\mu_0,\xi_0}),	
\end{align}	
where $l_0$ is given by \eqref{B:lm}. Since
\begin{align*}
\int_{\R^N}l_0\qty(\phi\qty[\mu_0,\xi_0]) Z_{j;\mu_0,\xi_0}
=
\int_{\R^N}l_0\qty(Z_{j;\mu_0,\xi_0}) \phi\qty[\mu_0,\xi_0],
\end{align*}
by \eqref{B:08}, we get
\begin{align}
\label{B:09}
\int_{\R^N}l_0\qty(\phi\qty[\mu_0,\xi_0]) Z_{j;\mu_0,\xi_0}
=
\int_{\R^N}\Lscr_{\mu_0,\xi_0}\phi\qty[\mu_0,\xi_0]\bar{Z}_{j,0;\mu_0,\xi_0},	
\end{align} 
where we used the fact that
\begin{align*}
\int_{\R^N}\Lscr_{\mu_0,\xi_0}\phi\qty[\mu_0,\xi_0]\bar{Z}_{j,0;\mu_0,\xi_0}
=
\int_{\R^N}\Lscr_{\mu_0,\xi_0}\bar{Z}_{j,0;\mu_0,\xi_0}\phi\qty[\mu_0,\xi_0].
\end{align*}

Next, by inserting \eqref{b:prob0} into \eqref{B:09}, we get
\begin{align*}
\int_{\R^N}l_0\qty(\phi\qty[\mu_0,\xi_0]) Z_{j;\mu_0,\xi_0}
=
\int_{\R^N}
g
\bar{Z}_{j,0;\mu_0,\xi_0}
-
\sum_{j=0}^{N}
\dfrac{\displaystyle\int_{\R^N}g Z_{j;\mu_0,\xi_0}}{\displaystyle\int_{\R^N}H_jZ_j }
\int_{\R^N}H_{j;\mu_0}^{\xi_0}\bar{Z}_{j,0;\mu_0,\xi_0},
\end{align*}
which, together with \eqref{tHbZ:id}, implies that
\begin{align*}
\int_{\R^N}l_0\qty(\phi\qty[\mu_0,\xi_0]) Z_{j;\mu_0,\xi_0}
=
\int_{\R^N}
g
\bar{Z}_{j,0;\mu_0,\xi_0}
+
\sum_{j=0}^{N}
\dfrac{\displaystyle\int_{\R^N}g Z_{j;\mu_0,\xi_0}}{\displaystyle\int_{\R^N}H_jZ_j }
\int_{\R^N}\widetilde{H}_{j,0;\xi_0}^{\mu_0}{Z}_{j;\mu_0,\xi_0}.
\end{align*}
Hence,
\begin{align*}
\int_{\R^N}
\qty(
l_0\qty(\phi\qty[\mu_0,\xi_0])
-r_0\qty(g)
-t_0\qty(g)
)Z_{j;\mu_0,\xi_0}=0,
\quad j=0,1,\ldots,N,
\end{align*}
by \Cref{prop:linsltn:prmt}, 
there exists a unique $\widetilde{\psi}_m$ solves \eqref{B:psi:eq}. Moreover, 
\begin{align*}
\norm{\widetilde{\psi}_0}_{X_{\mu_0,\xi_0}}\leq 
\frac{C}{\mu_0}
\norm{h_0\qty(\phi\qty[\mu_0,\xi_0])
	-r_0\qty(g)
	-t_0\qty(g)}_{Y_{\mu_0,\xi_0}}.
\end{align*}
By a similar argument as that appeared in the proof of \Cref{B:lemm:c}, we get
\begin{align*}
\norm{\widetilde{\psi}_0}_{X_{\mu_0,\xi_0}}\leq 
\frac{C}{\mu_0}
\norm{g}_{Y_{\mu_0,\xi_0}}.	
\end{align*}

The cases $m=1,2,\ldots,N$ can be obtained in the same manner. This ends the proof of \Cref{B:lemm:psi:td}
\end{proof}

Next, let us define
\begin{align}
\label{B:psim}
\psi_m 
=\widetilde{\psi}_m
-
\sum_{j=0}^N
\frac{1}{\mu_0}
\dfrac{\displaystyle\int_{\R^N}\phi\qty[\mu_0,\xi_0] \widetilde{H}_{j,m;\xi_0}^{\mu_0} }{\displaystyle\int_{\R^N} H_j Z_j}
Z_{j;\mu_0,\xi_0}.			
\end{align}

\begin{lemm}
Let $g\in Y_{\mu_0,\xi_0}$, $\phi\qty[\mu_0,\xi_0]$ and $\phi\qty[\mu,\xi]$ be the solutions to \eqref{b:prob0}
\eqref{b:prob0} respectively. Then $\phi\qty[\mu,\xi]$ is $C^1$ with respect to $\mu$ and $\xi$. More precisely, we have
\begin{gather}
\label{B:11}
\pdv{\phi}{\mu}\qty(\mu_0,\xi_0)=\psi_0, 
\\
\label{B:27}
\pdv{\phi}{\xi_m}\qty(\mu_0,\xi_0)=\psi_m, 
~~m=1,2,\ldots,N,
\end{gather}
and
\begin{gather}
\label{B:26}
\norm{\pdv{\phi}{\mu}\qty(\mu_0,\xi_0)}_{X_{\mu_0,\xi_0}}
\leq 
\frac{C}{\mu_0}
\norm{g}_{Y_{\mu_0,\xi_0}},	
\\
\label{B:28}
\norm{\pdv{\phi}{\xi_m}\qty(\mu_0,\xi_0)}_{X_{\mu_0,\xi_0}}
\leq 
\frac{C}{\mu_0}
\norm{g}_{Y_{\mu_0,\xi_0}},	\quad m=1,\ldots,N.		
\end{gather}
\end{lemm}
\begin{proof}
We only show that \eqref{B:11} and \eqref{B:26} holds.
The identity \eqref{B:27} and the estimates \eqref{B:28} can be proved in the same manner.

In order to show that \eqref{B:11} holds, we only need to prove that
\begin{align}
\label{B:25}
\lim\limits_{\mu\to\mu_0}
\norm{\frac{\phi\qty(\mu,\xi_0)-\phi\qty[\mu_0,\xi_0]}{\mu-\mu_0}-\psi_0  }=0.	
\end{align}

We claim that
\begin{gather}
\label{B:24}
\lim\limits_{\mu\to\mu_0}
\norm{\Pi^{\perp}\qty(\frac{\phi\qty(\mu,\xi_0)-\phi\qty[\mu_0,\xi_0]}{\mu-\mu_0})
	+\sum_{j=0}^N
	\frac{1}{\mu_0}
	\dfrac{\displaystyle\int_{\R^N}\phi\qty[\mu_0,\xi_0] \widetilde{H}_{j,m;\xi_0}^{\mu_0} }{\displaystyle\int_{\R^N} H_j Z_j}
	Z_{j;\mu_0,\xi_0}
}=0,
\\
\label{B:23}
\lim\limits_{\mu\to\mu_0}
\norm{\Pi\qty(\frac{\phi\qty(\mu,\xi_0)-\phi\qty[\mu_0,\xi_0]}{\mu-\mu_0})-\widetilde{\psi}_0  }=0.
\end{gather}
Assuming that \eqref{B:23} and \eqref{B:24} hold, 
using \eqref{B:psim} and the fact that 
\begin{align*}
\phi\qty[\mu,\xi]-\phi\qty[\mu_0,\xi_0]=
\Pi_{\mu_0,\xi_0}\qty(\phi\qty[\mu,\xi]-\phi\qty[\mu_0,\xi_0])
+
\Pi_{\mu_0,\xi_0}^{\perp}\qty(\phi\qty[\mu,\xi]-\phi\qty[\mu_0,\xi_0]),	
\end{align*}
we conclude that \eqref{B:25} holds. Moreover, by \Cref{B:lemm:psi:td},
we deduce that \eqref{B:26} holds. 

The rest of the proof is devoted to show that \eqref{B:24} and \eqref{B:23} hold.

\subsection*{Proof of \eqref{B:24}}
Since 
\begin{align*}
\Pi^{\perp}\qty(\frac{\phi\qty(\mu,\xi_0)-\phi\qty[\mu_0,\xi_0]}{\mu-\mu_0})
=
\sum_{j=0}^N\dfrac{\displaystyle\int_{\R^N}
	\qty(\phi\qty[\mu,\xi]-\phi\qty[\mu_0,\xi_0])	
	H_{j;\xi_0}^{\mu_0}  }{
\qty(\mu-\mu_0)
	\displaystyle\int_{\R^N}H_j Z_{j}} Z_{j;\mu_0,\xi_0},
\end{align*}
using $\int_{\R^N}{\phi\qty[\mu_0,\xi_0]}{H_{j;\xi_0}^{\mu_0}}=0$
and
$\int_{\R^N}{\phi\qty[\mu,\xi]}H_{j;\xi}^{\mu}=0$~ ($j=0,1,\ldots,N$),
we have
\begin{align*}
\Pi_{\mu_0,\xi_0}^{\perp}\qty(\phi\qty[\mu,\xi]-\phi\qty[\mu_0,\xi_0])
=
-\sum_{j=0}^N\dfrac{\displaystyle\int_{\R^N}
	\phi\qty[\mu,\xi]\qty(H_{j;\xi}^{\mu}-  H_{j;\xi_0}^{\mu_0})}{
	\displaystyle\int_{\R^N}H_j Z_{j}} Z_{j;\mu_0,\xi_0}.		
\end{align*}
Therefore, using the facts
$\norm{Z_{j;\mu_0,\xi_0}}_{X_{\mu_0,\xi_0}}\leq C$~($j=0,1,\ldots,N$),
 in order to show that \eqref{B:24}, we only need to prove
\begin{align}
\label{B:29}
\lim\limits_{\mu\to\mu_0}\qty(\int_{\R^N}
\phi\qty[\mu,\xi]\frac{H_{j;\xi}^{\mu}-  H_{j;\xi_0}^{\mu_0}}{\mu-\mu_0}
-\frac{1}{\mu_0}
\int_{\R^N}\phi\qty[\mu_0,\xi_0] \widetilde{H}_{j,m;\xi_0}^{\mu_0})=0.
\end{align}
Indeed, by \eqref{Hm:expand}, we have
\begin{align*}
\frac{H_{j;\xi}^{\mu}\qty(x)-  H_{j;\xi_0}^{\mu_0}\qty(x)}{\mu-\mu_0}
-\frac{1}{\mu_0}\widetilde{H}_{j,m;\xi_0}^{\mu_0}\qty(x)
=
U_{\mu_0,\xi_0}^{\frac{N+2}{N-2}}\qty(x)
\O\qty( \frac{\abs{\mu-\mu_0}}{\mu_0^2}),
\end{align*}
which, together with \eqref{b:cvg}, implies that \eqref{B:29} holds. Therefore, we proved \eqref{B:24}.
\subsection*{Proof of \eqref{B:23}}
First, by \eqref{B:07}, we have
\begin{align*}
\begin{dcases}
\Lscr\Pi_{\mu_0,\xi_0}\frac{\phi\qty(\mu,\xi_0)-\phi\qty[\mu_0,\xi_0]}{\mu-\mu_0}
=
\frac{\delta_{\mu,\xi_0}l\qty(\phi\qty(\mu,\xi_0))}{\mu-\mu_0}
-
\frac{\delta_{\mu,\xi_0}r\qty(g)}{\mu-\mu_0}
+
\frac{\delta_{\mu,\xi_0}t\qty(g)}{\mu-\mu_0}, 
\\
\int_{\R^N}	\Pi_{\mu_0,\xi_0}\qty(\frac{\phi\qty(\mu,\xi_0)-\phi\qty[\mu_0,\xi_0]}{\mu-\mu_0})H_{j;\xi_0}^{\mu_0}=0,	
\end{dcases}						
\end{align*}
where
\begin{gather}
\label{B:d0l}
\begin{align}
	\notag
	\delta_{\mu,\xi_0}l\qty(\phi\qty(\mu,\xi_0))
	=
	\alpha p \qty(\Riesz{U_{\mu,\xi}^{p-1}\phi\qty[\mu,\xi]}U_{\mu,\xi}^{p-1}
	- \Riesz{U_{\mu_0,\xi_0}^{p-1}\phi\qty[\mu,\xi]}U_{\mu_0,\xi_0}^{p-1})
	\\
	+
	\alpha (p-1)
	\qty( \Riesz{U_{\mu,\xi}^{p}}U_{\mu,\xi}^{p-2}\phi\qty[\mu,\xi]
	- \Riesz{U_{\mu_0,\xi_0}^{p}}U_{\mu_0,\xi_0}^{p-2}\phi\qty[\mu,\xi]),				
\end{align}
\\
\label{B:d0r}
\delta_{\mu,\xi}r\qty(g)
=
\sum_{j=0}^{N}
\dfrac{\displaystyle\int_{\R^N}g Z_{\mu,\xi,j}}{\displaystyle\int_{\R^N}{H_jZ_{j}}}
\qty(H_{j;\xi}^{\mu}-H_{j;\xi_0}^{\mu_0}),
\\
\label{B:d0t}
\delta_{\mu,\xi}t\qty(g)
=
-\sum_{j=0}^{N}
\dfrac{{\displaystyle\int_{\R^N}g \qty(Z_{j;\mu,\xi}-Z_{j;\mu_0,\xi_0})}}{\displaystyle\int_{\R^N}{H_jZ_{j}}}H_{j;\xi_0}^{\mu_0}.		
\end{gather}

Now, we claim
\begin{gather}
\label{B:19}
\norm{
	\frac{\delta_{\mu,\xi_0}l\qty(\phi\qty(\mu,\xi_0))}{\mu-\mu_0}
	-
	\frac{1}{\mu_0}l_0\qty(\phi\qty[\mu_0,\xi_0])
}_{Y_{\mu_0,\xi_0}}
=
\norm{g}_{Y_{\mu_0,\xi_0}}
\O\qty( \frac{\abs{\mu-\mu_0}}{\mu_0^2}),
\\
\label{B:20}
\norm{
	\frac{\delta_{\mu,\xi}r\qty(g)}{\mu-\mu_0}
	-
	\frac{1}{\mu_0}r_0\qty(g)	
}_{Y_{\mu_0,\xi_0}}
=
\norm{g}_{Y_{\mu_0,\xi_0}}
\O\qty( \frac{\abs{\mu-\mu_0}}{\mu_0^2}),
\\
\label{B:21}
\norm{
	\frac{\delta_{\mu,\xi}t\qty(g)}{\mu-\mu_0}
	-
	\frac{1}{\mu_0}t_0\qty(g)	
}_{Y_{\mu_0,\xi_0}}
=
\norm{g}_{Y_{\mu_0,\xi_0}}
\O\qty( \frac{\abs{\mu-\mu_0}}{\mu_0^2}).	
\end{gather}
\subsubsection*{Proof of \eqref{B:19}}

Indeed, by \eqref{B:14}, \eqref{B:13} and \eqref{B:12}, we have
\begin{align}
\label{B:15}
\phi\qty(\mu,\xi_0)\qty(x)
=
\phi\qty[\mu_0,\xi_0]\qty(x)
+
U_{\mu_0,\xi_0}\qty(x)\norm{g}_{Y_{\mu_0,\xi_0}}
\O\qty( \frac{\abs{\mu-\mu_0}}{\mu_0}).
\end{align}
Using \eqref{U:expand}, we deduce that
\begin{align}
\notag
\label{B:16}
&U_{\mu,\xi_0}^{p-1}\qty(x)-U_{\mu_0,\xi_0}^{p-1}\qty(x)
\\
=
&\qty(p-1)U_{\mu_0,\xi_0}^{p-2}\qty(x)Z_{0;\mu_0,\xi_0}\qty(x)
\frac{\mu-\mu_0}{\mu_0}
+
U_{\mu_0,\xi_0}^{p-1}\qty(x)\O\qty( \frac{\abs{\mu-\mu_0}^2}{\mu_0^2}).
\end{align}
By \eqref{B:15} and \eqref{B:16}, using \eqref{A:ab-ineq} and \eqref{B:g:phi:bd}, we get 
\begin{align}
\label{B:17}
\notag
&\Riesz{U_{\mu,\xi}^{p-1}\phi\qty[\mu,\xi]}\qty(x)U_{\mu,\xi}^{p-1}\qty(x)
- \Riesz{U_{\mu_0,\xi_0}^{p-1}\phi\qty[\mu,\xi]}\qty(x)U_{\mu_0,\xi_0}^{p-1}\qty(x)\\
\notag
=&
\qty(p-1) \Riesz{U_{\mu_0,\xi_0}^{p-1}\phi\qty[\mu_0,\xi_0]}\qty(x)
U_{\mu_0,\xi_0}^{p-2}\qty(x)Z_{0;\mu_0,\xi_0}\qty(x)
\frac{\mu-\mu_0}{\mu_0}
\\
\notag
&+
\qty(p-1) \Riesz{U_{\mu_0,\xi_0}^{p-2}Z_{0;\mu_0,\xi_0}\phi\qty[\mu_0,\xi_0]}\qty(x)
U_{\mu_0,\xi_0}^{p-1}\qty(x)
\frac{\mu-\mu_0}{\mu_0}
\\
&+
\Riesz{U_{\mu_0,\xi_0}^p}\qty(x)U_{\mu_0,\xi_0}^{p-1}\qty(x)\norm{g}_{Y_{\mu_0,\xi_0}}
\O\qty( \frac{\abs{\mu-\mu_0}^2}{\mu_0^2}).
\end{align}

By a similar argument as above, we can obtain that
\begin{align}
\label{B:18}
\notag
&\Riesz{U_{\mu,\xi}^{p}}\qty(x)U_{\mu,\xi}^{p-2}\qty(x)\phi\qty[\mu,\xi]\qty(x)
- \Riesz{U_{\mu_0,\xi_0}^{p}}\qty(x)U_{\mu_0,\xi_0}^{p-2}\qty(x)\phi\qty[\mu,\xi]\qty(x)
\\
\notag
=&
p\Riesz{U_{\mu_0,\xi_0}^{p-1}Z_{0;\mu_0,\xi_0}}\qty(x)U_{\mu_0,\xi_0}^{p-2}\qty(x)\phi\qty[\mu_0,\xi_0]\qty(x)
\\
\notag
&+
(p-2)\Riesz{U_{\mu_0,\xi_0}^{p}}\qty(x)U_{\mu_0,\xi_0}^{p-3}\qty(x)Z_{0;\mu_0,\xi_0}\qty(x)\phi\qty[\mu_0,\xi_0]\qty(x)
\\
&+
\Riesz{U_{\mu_0,\xi_0}^p}\qty(x)U_{\mu_0,\xi_0}^{p-1}\qty(x)\norm{g}_{Y_{\mu_0,\xi_0}}
\O\qty( \frac{\abs{\mu-\mu_0}^2}{\mu_0^2}).
\end{align}
By inserting \eqref{B:17} and \eqref{B:18} into \eqref{B:d0l}, we obtain that \eqref{B:19} holds.
\subsubsection*{Proof of \eqref{B:20}}
Indeed, by \eqref{Hm:expand}, we have
\begin{align}
\label{B:22}
\norm{
H_{j;\xi_0}^{\mu}-H_{j;\xi_0}^{\mu_0}
-\widetilde{H}^{\mu_0}_{m,0;\xi_0}\qty(x)\frac{\mu-\mu_0}{\mu_0}
}_{Y_{\mu_0,\xi_0}}
=
\O\qty( \frac{\abs{\mu-\mu_0}^2}{\mu_0^2}).
\end{align}
Using \eqref{Zj:expand}, we get
\begin{align*}
\int_{\R^N}g Z_{\mu_0,\xi,j}
=
\int_{\R^N}g Z_{\mu_0,\xi_0,j}
+
\norm{g}_{Y_{\mu_0,\xi_0}}\O\qty( \frac{\abs{\mu-\mu_0}}{\mu_0}),
\end{align*}
which, together with \eqref{B:22}, implies that \eqref{B:20} holds.
\subsubsection*{Proof of \eqref{B:21}}
By \eqref{Zj:expand}, we get
\begin{align*}
\norm{Z_{j;\mu,\xi}
-
Z_{j;\mu_0,\xi_0} \qty(x)
- \bar{Z}_{j,0;\mu_0,\xi_0}\qty(x)\frac{\mu-\mu_0}{\mu_0}}
=
\O\qty( \frac{\abs{\mu-\mu_0}^2}{\mu_0^2}),
\end{align*}
which, together with the facts $\norm{H_{j;\xi_0}^{\mu_0}}_{Y_{\mu_0,\xi_0}}\leq C$, implies that \eqref{B:21} holds.	
\end{proof}

\end{document}